\documentclass[draft]{amsart}
\usepackage{amsfonts,amssymb,stmaryrd}
\theoremstyle{plain}
\newtheorem*{theorem*}{Theorem}
\newtheorem*{lemma*}{Lemma}
\newtheorem{corollary*}{Corollary}
\newtheorem*{proposition*}{Proposition}
\newtheorem{conjecture*}{Conjecture}
\newtheorem{theorem}{Theorem}[section]
\newtheorem{lemma}[theorem]{Lemma}
\newtheorem{corollary}[theorem]{Corollary}
\newtheorem{proposition}[theorem]{Proposition}

\newtheorem{question}[theorem]{Question}
\newtheorem{ther}{Theorem}

\theoremstyle{remark}

\newtheorem*{remark}{Remark}
\newtheorem*{remarks}{Remarks}
\newtheorem*{definition}{Definition}
\newtheorem*{notation}{Notation}
\newtheorem{example}[theorem]{Example}
\newtheorem{examples}[theorem]{Examples}
\newtheorem*{claim}{Claim}

\theoremstyle{definition}

   \def\Z{\Bbb{Z}} \def\R{\Bbb{R}} 
\def\N{\Bbb{N}}  
 
    \def\bp{\begin{pmatrix}}
 \def\ep{\end{pmatrix}} \def\bn{\begin{enumerate}} 
   \def\en{\end{enumerate}}
\def\ba{\begin{array}} \def\ea{\end{array}}  
     
\def\id{\operatorname{id}}   
  
\def\be{\begin{equation}} \def\ee{\end{equation}}

     \def\fr12{\frac{1}{2}} \def\z12{\Z[\fr12]}

\def\ol{\overline}

\def\epi{\text{\rm epi}}

\newcommand\conv{\operatorname{conv}}
\newcommand\cl{\operatorname{cl}}
\newcommand\bd{\operatorname{bd}}
\newcommand\interior{\operatorname{int}}
\newcommand\graph{\operatorname{graph}}

\newcommand{\abs}[1]{\lvert#1\rvert}
\newcommand{\norm}[1]{\left\lVert#1\right\rVert}
\newcommand{\dom}{\operatorname{dom}}
\newcommand{\trans}{{\operatorname{t}}}

\DeclareMathAlphabet{\mathbf}{OML}{cmm}{b}{it}

\numberwithin{equation}{section}

\begin{document}

\title{Definable Versions of Theorems by Kirszbraun and Helly}


\date{July 2010}
\begin{abstract}
Kirszbraun's Theorem states that every Lipschitz map $S\to\R^n$, where $S\subseteq \R^m$, has an extension to a Lipschitz map $\R^m\to\R^n$ with the same Lipschitz constant. Its proof relies on Helly's Theorem: every family of compact subsets of $\R^n$, having the property that each of its subfamilies consisting of at most $n+1$ sets share a common point, has a non-empty intersection.
We prove versions of these theorems valid for definable maps and sets in arbitrary definably complete expansions of ordered fields.
\end{abstract}

\author{Matthias Aschenbrenner}
\address{University of California, Los Angeles, California, USA}
\email{matthias@math.ucla.edu}

\author{Andreas Fischer}
\address{The Fields Institute, Toronto, Ontario, Canada}
\curraddr{Gymnasium St. Ursula Dorsten, 46282 Dorsten, Germany}
\email{el.fischerandreas@live.de}

\maketitle

\section*{Introduction}

\noindent
Let $L$ be a non-negative real number and let $f\colon S\to \R^n$, $S\subseteq\R^m$, be an $L$-Lipschitz map, i.e., $||f(x)-f(y)||\leq L\,||x-y||$ for all $x,y\in S$. 
It was noted by McShane and Whitney independently (1934) that if $n=1$, then $f$ extends to an $L$-Lipschitz function $\R^m\to\R$.  This immediately implies that for general $n$, there always exists a $\sqrt{n}\,L$-Lipschitz map $F\colon\R^m\to\R^n$ with $F|S=f$.  A seminal result proved by Kirszbraun (1934) shows that in fact, the multiplicative constant $\sqrt{n}$ is redundant: {\it there is an $L$-Lipschitz map $F\colon\R^m\to \R^n$  such that $F|S=f$.}\/ 
This theorem plays an important role in geometric measure theory (see \cite{Federer}) and has been generalized in many ways, e.g., to more general moduli of continuity and arbitrary Hilbert spaces (see \cite[Theorem~1.12]{BenLi}).
The usual proofs of theorems of this kind in the literature employ, in some form or other, the Axiom of Choice. (See, e.g., \cite{BenLi, DGK, Federer, Goebel, Heinonen}.) This prompted Chris Miller to ask:  {\it suppose $f$ as before is semialgebraic; is there a semialgebraic $L$-Lipschitz map $\R^m\to\R^n$ extending $f$?}\/ More generally:
\begin{quote}
{\it Let $\mathfrak R$ be an o-minimal expansion of a real closed ordered field $R$, and let $f\colon S\to R^n$, $S\subseteq R^m$, be definable in $\mathfrak R$ and $L$-Lipschitz \textup{(}where $L\in R$, $L\geq 0$\textup{)}. Does $f$ admit an extension to an $L$-Lipschitz map $R^m\to R^n$ which is definable in $\mathfrak R$?}
\end{quote}
Here and below, ``definable'' means ``definable, possibly with parameters.'' 
Questions like these are of interest since many (but not all \cite{HP}) properties familiar from real analysis and topology hold for sets and functions definable in o-minimal structures, even if the underlying ordered set is different from the real line. See \cite{vdDries-Tame} for this, and basic definitions concerning o-minimal structures. 

It is easy to see that the question above has a positive answer in the case $n=1$ by the McShane-Whitney construction alluded to above (see Proposition~\ref{prop:McShane-Whitney} below) and also if the domain $S$ of $f$ is convex (see Proposition~\ref{prop:Kirszbraun convex}). In this paper we answer Miller's question positively in general. In fact, o-minimality may be replaced by a weaker assumption. For the rest of this introduction, we fix
an expansion $\mathfrak R=(R,0,1,{+},{\,\cdot\,},{<},\dots)$ of a real closed ordered field, and ``definable'' means ``definable in $\mathfrak R$.'' One says that $\mathfrak R$ is {\bf definably complete} if every non-empty definable subset of $R$ which is bounded from above has a least upper bound in $R$. (See Section~\ref{sec:defcompleteness} below for more on this notion.) Our first main result is the following:

\begin{ther}\label{thm:Kirszbraun}
Suppose $\mathfrak{R}$ is definably complete. Let  $L\in R$, $L\geq 0$, and let $f\colon S\rightarrow R^n$, where $S\subseteq R^m$, be a definable $L$-Lipschitz map. Then there exists a definable $L$-Lipschitz map $F\colon R^m\to R^n$ such that $F|S=f$.
\end{ther}

It turns out that definable completeness is indeed necessary for the conclusion of Theorem~\ref{thm:Kirszbraun} to hold, see Proposition~\ref{prop:def completeness necessary} below. The extension $F$ of $f$ in the theorem can additionally be chosen to depend uniformly on parameters, see Corollary~\ref{cor:uniform Kirszbraun}.

The proof of Theorem~\ref{thm:Kirszbraun} is based on a recent constructive approach to Kirsz\-braun's Theorem due to Bauschke and Wang \cite{Bauschke1, Bauschke2} using the proximal average of convex functions. 
This is the culmination of a long train of thought (going back at least to Minty \cite{Minty}) relating Lipschitz maps to monotone set-valued maps.
It is remarkable that the arguments of loc.~cit.~may be transferred in a straightforward way to the setting of definable complete expansions of ordered fields, with the exception of an interesting property of definable families: In general,  a family $\mathcal C$ of closed balls in $R^n$ with the finite intersection property may have empty intersection; however (and perhaps, somewhat surprisingly), if $\mathfrak R$ is definably complete and the family $\mathcal C$ is definable, then $\bigcap\mathcal C\neq\emptyset$. More precisely, we have the following result:

\begin{ther}\label{thm:Helly}
Suppose $\mathfrak R$ is definably complete.
Let $\mathcal C$ be a definable family of closed bounded convex subsets of $R^n$. If any collection of at most $n+1$ sets from $\mathcal C$ has a non-empty intersection, then $\mathcal C$ has a non-empty intersection.
\end{ther}

This theorem is a definable analogue of a classical theorem of Helly (1913) on families of compact convex subsets of $\R^n$. In the standard proofs of this theorem (e.g., as given in \cite{Webster}), one first reduces to the case of a finite family by a topological compactness argument, which is unavailable in the more general context considered here. Thus we were forced to find a different proof which adapts to infinite definable families. (See \cite{DGK, Eckhoff} for the history and numerous variants of Helly's Theorem.)

Note that the theorem fails trivially if the assumptions ``closed'' or ``bounded'' are dropped, as suitable definable families of intervals in $R$ show. 
Definable completeness of $\mathfrak R$ is also necessary in this case:  if $\mathfrak R$ has the property that every infinite definable family of closed bounded convex subsets of $R$ with empty intersection contains two disjoint members,  then $\mathfrak R$ is definably complete. 
It may also be worth noting that the natural definable analogue of the Heine-Borel Theorem (a definable set $S\subseteq R^n$ is closed and bounded if and only if every definable family of closed subsets of $S$ with the finite intersection property has a non-empty intersection) fails if the ordered field $R$ is non-archimedean. (See Section~\ref{sec:related results}.)

\subsection*{Organization of the paper}
Many of the basic properties of convex sets in $\R^n$ (as presented in, say, \cite{Schneider,Webster}) hold in the setting of a definably complete expansion of an ordered field, provided attention is restricted to definable convex sets.  After a preliminary Section~\ref{sec:Preliminaries}, we develop some of these properties in Section~\ref{sec:convex}, restricting ourselves to what is necessary for the proof of Theorems~\ref{thm:Kirszbraun} and \ref{thm:Helly}. We give the proof of Theorem~\ref{thm:Helly} and some applications of this theorem in Section~\ref{sec:Helly proof}.  In Section~\ref{sec:related results} we also present another proof of Theorem~\ref{thm:Helly} valid in the case where $\mathfrak R$ is o-minimal, due to S.~Starchenko (and based on results by Dolich and Peterzil-Pillay). In Section~\ref{sec:convex analysis} we  establish a few basic results of convex analysis in the definably complete setting, and in Section~\ref{sec:Kirszbraun proof} we prove Theorem~\ref{thm:Kirszbraun}. In Section~\ref{sec:Variants} we discuss some variants of Theorem~\ref{thm:Kirszbraun}: a weak version of Kirszbraun's Theorem for Lipschitz maps which are locally definable in expansions of the ordered field of real numbers, and the extension problem for uniformly continuous definable maps.

\subsection*{Acknowledgments}
We thank Chris Miller for many discussions around the topics of this paper, and Sergei Starchenko for permission to include the argument in Section~\ref{sec:related results}. This paper was partially written while both authors were participating in the thematic program on O-minimal Structures and Real Analytic Geometry at the Fields Institute in Toronto in 2009. The support of this institution is gratefully acknowledged. The first author was also partially supported by a grant from the National Science Foundation.

\subsection*{Conventions and notations} We let $k$, $m$, $n$, range over the set $\N=\{0,1,2,\dots\}$ of natural numbers. ``Definable'' means ``definable, possibly with parameters.''

Let $R$ be a real closed ordered field. We equip $R$ with the order topology, and each $R^n$ with the corresponding product topology. Given a subset $S$ of $R^n$ we write $\interior(S)$ for the interior, $\cl(S)$ for the closure, and $\bd(S)=\cl(S)\setminus\interior(S)$ for the boundary of $S$. We write the dot product of $x=(x_1,\dots,x_n)\in R^n$ and $y=(y_1,\dots,y_n)\in R^n$ as
$$\langle x,y\rangle = x_1y_1+\cdots+x_ny_n,$$
and we set
$||x|| := \sqrt{ \langle x,x \rangle }$. For $\varrho>0$ and $x\in R^n$ we write
$$B_\varrho(x):=\big\{y\in R^n: ||x-y||<\varrho\big\}, \qquad
  \ol{B}_\varrho(x):=\big\{y\in R^n: ||x-y||\leq\varrho\big\}$$
for the open respectively closed ball in $R^n$ with radius $\varrho$ and center $x$. A set $S\subseteq R^n$ is said to be bounded if $S\subseteq B_\varrho(0)$ for some $\varrho>0$.

For $a,b\in R$ we put $[a,b]:=\{x\in R:a\leq x\leq b\}$. For $S\subseteq R$ and $a\in R$ we set
$S^{>a}:=\{r\in S:r>a\}$ and similarly with other inequality symbols in place of ``$>$.''
We extend the linear ordering of $R$ to a linear ordering of $R_{\pm\infty}=R\cup \{-\infty,+\infty\}$ such that $-\infty<R<+\infty$. 
We assume the usual rules for addition and multiplication with $\pm\infty$. 
We also set $R_\infty=R\cup\{+\infty\}$.
We say that a function $f\colon S\to R_{\pm\infty}$ (where $S\subseteq R^n$) is {\bf finite} at $x\in S$ if $f(x)\in R$, and we simply say that $f$ is finite if it is finite at every $x\in S$.

\section{Preliminaries}\label{sec:Preliminaries}

\noindent
This section contains material which is fundamental for the following sections. In Sections~\ref{sec:defcompleteness} and \ref{sec:def BW} we collect basic properties of definably complete expansions of ordered fields. In Section~\ref{sec:Lipschitz} we discuss Lipschitz maps, and Section~\ref{sec:Minkowski} contains a useful fact about Minkowski sums of closed sets.

\subsection{Definable completeness}\label{sec:defcompleteness}
Let $\mathfrak R$ be an expansion of an ordered field $R$.
One says that $\mathfrak R$ is {\bf definably complete} if every non-empty definable subset of $R$ which is bounded from above has a least upper bound in $R$. Clearly then every non-empty definable subset of $R$ which is bounded from below has a greatest lower bound in $R$. Moreover, if $\mathfrak R$ is definably complete, then the field $R$ is necessarily real closed. For a proof of this fact see \cite{Miller}, where further
basic properties of definably complete structures were developed. In particular, the following characterization of definable completeness is proved in \cite[Corollary~1.5]{Miller}. 

\begin{proposition}
The following are equivalent:
\begin{enumerate}
\item $\mathfrak R$ is definably complete.
\item Every continuous definable function $f\colon [a,b]\to R$ has the intermediate value property: for each $y\in R$ between $f(a)$ and $f(b)$ there is some $x\in [a,b]$ with $y=f(x)$.
\item Intervals in $R$ are definably connected.
\item $R$ is definably connected.
\end{enumerate}
\end{proposition}

(Recall that a set $S\subseteq R^n$ is said to be definably connected if for all definable open sets $U,V\subseteq R^n$ with $S=(S\cap U)\cup (S\cap V)$ and $S\cap U\cap V=\emptyset$, we have $S\subseteq U$ or $S\subseteq V$.)

\medskip

The notion of
definable completeness is intended to capture the first-order content of Dedekind completeness: indeed, every expansion of the ordered field of real numbers is definably complete, and every structure elementarily equivalent to a definably complete structure is definably complete \cite[Section~3]{Miller}. 
Definable completeness is connected to o-minimality:
If $\mathfrak R$ is o-minimal, then $\mathfrak R$ is definably complete. (In fact, it is enough to require that the open core $\mathfrak R^\circ$ of $\mathfrak R$ is o-minimal.) If $\mathfrak R$ is o-minimal, and $R'$ is a proper dense subset of $R$ which is the underlying set of an elementary substructure of $\mathfrak R$, then $(\mathfrak R,R')$ is definably complete, by \cite{vdDries-Dense}.
However, definable completeness is sufficiently far removed from o-minimality to warrant independent interest:  by results in \cite{DM, MS}, $\mathfrak R$ is o-minimal if and only if
$\mathfrak R$ is definably complete, every definable subset of $R$ is constructible (i.e., a finite boolean combination of open sets), and there is no definable subset of $R$ which is both infinite and discrete.

\medskip

{\it In the rest of this section we assume that $\mathfrak R$ is definably complete.}

\begin{notation}
We say that $A\subseteq R_{\pm\infty}$ is definable if $A\cap R$ is definable. With this convention, every non-empty definable subset $A$ of  $R_{\pm\infty}$ has a least upper bound in $R_{\pm\infty}$, which we denote by $\sup A$, and $A$ has a greatest lower bound in $R_{\pm\infty}$, denoted by $\inf A$.
We also set $\sup\emptyset:=-\infty$ and $\inf\emptyset:=+\infty$. 
\end{notation}

We have a weak version of definable choice \cite[Proposition~1.8]{Miller}:

\begin{lemma}\label{lem:defchoice}
Let $\mathcal C=\{C_a\}_{a\in A}$, where $A\subseteq R^n$, be a definable family of non-empty closed and bounded subsets of $R^m$. Then there is a definable map $f\colon A\to R^m$ such that $f(a)\in C_a$ for every $a\in A$.
\end{lemma}

Many facts familiar from set-theoretic topology in $\R$ continue to hold for $\mathfrak R$, provided attention is restricted to the definable category. In the following we collect some of those properties. 
The first one \cite[Lemma~1.9]{Miller} (which follows from Lemma~\ref{lem:defchoice}) captures a crucial feature of compact subsets of $\R^n$:

\begin{lemma}\label{lem:monotone}
Let $\mathcal C=\{C_a\}_{a\in A}$, where $A\subseteq R$, be a definable family of non-empty closed bounded subsets of $R^n$ which is monotone, i.e., either $C_a\subseteq C_b$ for all $a,b\in A$ with $a\leq b$, or $C_a\supseteq C_b$ for all $a,b\in A$ with $a\leq b$. Then $\bigcap\mathcal C\neq\emptyset$.
\end{lemma}

Note that this lemma implies a special case of Theorem~\ref{thm:Helly} for monotone definable families of closed bounded sets (without the assumption of convexity).

\begin{proposition}\label{prop:cbd}
Let $f\colon S\to R^n$ be definable and continuous, where $S\subseteq R^m$. If $S$ is closed and bounded, then so is $f(S)$.
\end{proposition}

This is \cite[Proposition~1.10]{Miller}.
As an immediate consequence, one has:

\begin{corollary}\label{cor:minmax}
Let $f\colon S\to R$ be definable and continuous, where $S\subseteq R^m$ is closed and bounded. Then $f$ achieves a minimum and a maximum on $S$.
\end{corollary}

\subsection{Definable Bolzano-Weierstrass Theorem}\label{sec:def BW}
For our investigations it is useful to have a counterpart of the
Bolzano-Weierstrass Theorem  from classical analysis, concerning infinite sequences in compact subsets of $\R^n$. 
In o-minimal geometry, this counterpart is described by the Curve~Selection~Lemma, which is not available in the definably complete situation.
%
%
%
\begin{definition}
Let $\gamma\colon I\to R^n$ be a definable function, where
$I\subseteq R^{>0}$ is unbounded. We call such a function $\gamma$  a {\bf sequence-map}. A sequence-map $\gamma'\colon I'\to R^n$ is said to be a {\bf subsequence-map} of $\gamma$ if $I'\subseteq I$ and $\gamma'=\gamma|I'$.
We say that $\gamma$ {\bf converges} if
$a=\lim\limits_{t\rightarrow \infty,\ t\in I} \gamma(t)$ exists, and in this case, we say that $\gamma$ converges to $a$.
An element $a$ of $R^n$ such that there is a subsequence-map $\gamma'$ of $\gamma$ converging to $a$ is called an {\bf accumulation point} of $\gamma$.
\end{definition}
\begin{proposition}[Definable Bolzano-Weierstrass Theorem]\label{prop:BW}
Let $S\subseteq R^n$ be a closed and bounded definable set.
Then every sequence-map $\gamma\colon I\to S$ has an accumulation point in $S$.
\end{proposition}
\begin{proof}
Let $\gamma\colon I\to S$ be a sequence-map. In the following let $\varepsilon$, $\varepsilon'$ and $t$ range over $R^{>0}$.
For every $t$ put $S_t:=\cl(\gamma(I^{>t}))$, a closed and bounded non-empty definable subset of $S$. By Lemma~\ref{lem:monotone} we have $\bigcap_t S_t\neq\emptyset$. Let $a\in \bigcap_t S_t$; we claim that $a$ is an accumulation point of $\gamma$. To see this, note that by choice of $a$, for every $\varepsilon$ the definable set
$$I_{\varepsilon} := \big\{t\in I : \norm{\gamma(t)-a}<\varepsilon \big\}$$
is unbounded, hence $I_\varepsilon\cap R^{\geq 1/\varepsilon}\neq\emptyset$. For each $\varepsilon$ put 
$$t_\varepsilon := \inf\big(I_\varepsilon\cap R^{\geq 1/\varepsilon}\big)\in R^{>0}, \qquad s_\varepsilon := t_\varepsilon+1.$$
So for every $\varepsilon$ there exists $t\in I_\varepsilon$ with $1/\varepsilon \leq t\leq s_\varepsilon$, hence the definable subset
$$I' := \big\{t: \exists\varepsilon\, (t\in I_\varepsilon\ \&\ t\leq s_\varepsilon)\big\}$$
of $I$ is unbounded. 
Moreover, if $\varepsilon'\leq\varepsilon$ then $s_{\varepsilon'}\geq s_{\varepsilon}$.
Let $\varepsilon$ be given, and let $t\in I'$ with $t > s_{\varepsilon}$. Then there is some $\varepsilon'$ with $t\in I_{\varepsilon'}$ and $t\leq s_{\varepsilon'}$. Then $\varepsilon>\varepsilon'$ and hence $t\in I_\varepsilon$. This shows that $a=\lim\limits_{t\rightarrow \infty,\ t\in I'} \gamma(t)$.
\end{proof}

\subsection{Moduli of continuity}
Let $f\colon S\to R^n$ be a definable map, where $S\subseteq R^m$ is non-empty. Then the {\bf mod\-u\-lus of continuity $\omega_f\colon R^{\geq 0}\to R_\infty$ of $f$} is given by
$$\omega_f(t):=\sup\big\{ ||f(x)-f(y)|| : x,y\in S,\ ||x-y||\leq t \big\}.$$
The function $\omega_f$ is definable and increasing with $\omega_f\geq 0$, and $f$ is uniformly continuous if and only if $\omega_f(t)\to 0$ as $t\to 0^+$.  If $f$ is bounded, then $\omega_f$ is finite.

\begin{lemma}\label{lem:extension to closure}
Suppose $f$ is uniformly continuous. Then $f$ extends uniquely to a continuous map $F\colon\cl(S)\to R^n$. This extension is again definable, with
$$\omega_f(t) \leq \omega_F(t)\leq \inf_{t'>0} \omega_f(t'+t)\qquad\text{for all $t>0$.}$$
In particular, $F$ remains uniformly continuous. 
\end{lemma}
\begin{proof}
Uniqueness is easy to see (and only needs continuity of $f$). For existence, take $\delta>0$ such that the restriction of $\omega_f$ to the interval $A:=(0,\delta)$ is finite. 
Let $x_0\in\cl(S)$; we introduce a definable family $\mathcal C=\mathcal C(x_0)$ as follows: For $t\in A$ let
$$C_t := \bigcap_{x\in \ol{B}_t(x_0)\cap S} \ol{B}_{\omega_f(t)}(f(x));$$
then $\mathcal C=\{C_t\}_{t\in A}$ is a decreasing definable family of
non-empty closed bounded subsets of $R^m$.  By Lemma~\ref{lem:monotone}, we have $\bigcap\mathcal C\neq\emptyset$. Note that this intersection is a singleton: if $y\neq y'$ are both in $\mathcal C$, take $t\in A$ such that $\omega_f(t)<\frac{1}{2}||y-y'||$; then for every $x\in\ol{B}_t(x_0)\cap S$ we have $||y-f(x)||\leq\omega_f(t)$ and $||y'-f(x)||\leq\omega_f(t)$, hence $||y-y'||\leq 2\omega_f(t)$, a contradiction. Therefore we have a definable map $F\colon\cl(S)\to R^m$ which sends $x_0\in\cl(S)$ to the unique element in $\bigcap \mathcal C(x_0)$. Clearly the map $F$ extends $f$, and hence $\omega_f\leq\omega_F$. Let $t>0$ and $x_0,x_1\in\cl(S)$ with $||x_0-x_1||\leq t$ be given. 
For every $t'$ with $0<t'<\delta-t$ we find $y_0,y_1\in S$ with $||x_0-y_0||\leq t'/2$ and $||x_1-y_1||\leq t'/2$, and  so $||y_0-y_1||\leq t+t'$; then 
\begin{align*}
||F(x_0)-F(x_1)||&\leq ||F(x_0)-f(y_0)||+||f(y_0)-f(y_1)||+||F(x_1)-f(y_1)||\\
&\leq \omega_f(t')+\omega_f(t'+t)+\omega_f(t').
\end{align*}
The inequality for the moduli of continuity now follows by letting $t'\to 0$.
\end{proof}

As over $\R$ we have  uniform continuity of definable continuous maps with closed and bounded domain, as shown in the next lemma. (The classical proof of this fact uses the finite subcover property of compact sets.) Notice that by Corollary~\ref{cor:minmax}, definable, closed and bounded non-empty subsets $D$ and $E$ of $R^m$ have a common point if and only if $d(D,E)=0$, where $d(D,E):=\inf\big\{\norm{x-y}:x\in D, y\in E\big\}$ is the distance between $D$ and $E$. 
\begin{lemma}\label{lem:cbd implies uniform continuous}
Suppose $S$ is closed and bounded, and $f$ is continuous. Then $f$ is uniformly continuous.
\end{lemma}
\begin{proof}
Note that $\omega_f$ is finite since $f$ is bounded, cf.~Proposition~\ref{prop:cbd}. We shall show that $\omega_f(t)\to 0$ as $t\to 0^+$.
Assume, for a contradiction, that $\varepsilon>0$ is such that
$\omega_f(t)\geq\varepsilon$ for arbitrarily small positive $t$.
Then $D:=\{(x,y)\in S\times S:\norm{f(x)-f(y)}\geq \varepsilon\}
$ and $E:=\{(x,x) : x\in S\}$ are
disjoint definable closed and bounded non-empty sets with $d(D,E)=0$,
a contradiction. 
\end{proof}

We say that a function $\omega\colon R^{\geq 0}\to R_\infty$ is {\bf a modulus of continuity of $f$} if $\omega_f\leq\omega$.
The following is easy to show; we skip the proof:

\begin{lemma}\label{lem:inf and sup of Lipschitz functions}
Let $\omega\colon R^{\geq 0}\to R^{\geq 0}$ be definable, and let $\{f_a\}_{a\in A}$ be a definable family of functions $f_a\colon S\to R$ with modulus of continuity $\omega$. If the function
$$x\mapsto \inf_{a\in A} f_a(x)\colon S\to R\cup\{-\infty\}$$
is finite at one point of $S$, then it is finite with modulus of continuity $\omega$.
Similarly, if the function
$$x\mapsto \sup_{a\in A} f_a(x)\colon S\to R\cup\{+\infty\}$$
is finite at one point of $S$, then it is finite with modulus of continuity $\omega$.
\end{lemma}

\subsection{Lipschitz maps}\label{sec:Lipschitz}
Let $f\colon S\to R^n$ be a definable map, where $S\subseteq R^m$ is non-empty.
Given $L\in R^{\geq 0}$, we say that $f$ is {\bf $L$-Lipschitz} if $||f(x)-f(y)||\leq L||x-y||$ for all $x,y\in S$. We call $f$ {\bf Lipschitz} if $f$ is $L$-Lipschitz for some $L\in R^{\geq 0}$. Every Lipschitz map is  uniformly continuous; in fact, given $L\in R^{\geq 0}$,
$f$ is $L$-Lipschitz if and only if $t\mapsto Lt$ is a modulus of continuity of $f$. Consequently, if $f$ is $L$-Lipschitz, then $f$ extends uniquely to a continuous map $\cl(S)\to R^n$, and this map is also $L$-Lipschitz, by Lemma~\ref{lem:extension to closure}.

We use  {\bf non-expansive} synonymously for $1$-Lipschitz. By the triangle inequality, for every $y\in R^n$ the function
$$x\mapsto d(x,y):=||x-y||\colon R^n\to R$$
is non-expansive. From Lemma~\ref{lem:inf and sup of Lipschitz functions} we  therefore obtain:

\begin{corollary} \label{cor:distance is non-expansive}
For every definable subset $S$ of $R^n$, the distance function
$$x\mapsto d(x,S):=\inf\big\{d(x,y):y\in S\big\}\colon R^n\to R$$
is non-expansive.
\end{corollary}

The set $S$ in the previous corollary was not assumed to be closed.
However, in this context one may often reduce to the case of a closed set, since $d(x,S)=d(x,\cl(S))$ for every definable set $S\subseteq R^n$ and every $x\in R^n$.
For closed sets we have, as a consequence of Corollary~\ref{cor:minmax}:

\begin{corollary}\label{cor:nearest point}
Suppose $S$ is closed and definable. Then
for every $x\in R^n$ there is a nearest point of $S$ to $x$, that is, a point $y_0\in S$ such that $d(x,y_0) = d(x,S)$.
\end{corollary}
\begin{proof}
Let $x\in R^n$.
Choose $\varrho>0$ such that the closed and bounded definable set $S\cap \ol{B}_\varrho(x)$ is non-empty. By Corollary~\ref{cor:minmax} the function $y\mapsto d(x,y)$ attains a minimum on this set, say at $y_0$; then $y_0$ is a nearest point of $S$ to $x$.
\end{proof}

The following concept plays an important role in the proof of Theorem~\ref{thm:Kirszbraun} below.

\begin{definition}
A map $f\colon S\rightarrow R^n$, where $S\subseteq R^n$, is called
{\bf  firmly non-expansive} if
$$\norm{f(x)-f(y)}^2\leq\langle f(x)-f(y),x-y\rangle
\qquad\text{for all $x,y\in S$.}$$
\end{definition}
The Cauchy-Schwarz Inequality implies that every firmly non-expansive map is non-expansive.  We also have the following fact, well-known in classical convex analysis  (see, e.g., \cite[Theorem~12.1]{Goebel}):
\begin{proposition}\label{prop44}
Let $S\subseteq R^n$. Then $f\mapsto \frac{1}{2} (f+\id)$
 is a bijection from the set of non-expansive maps $S\to R^n$ to the set of firmly non-expansive maps $S\to R^n$.
\end{proposition}
\begin{proof}
For $x,y\in S$ consider 
\begin{align*} 
a(x,y) &:=\textstyle\frac{1}{4}\norm{x-y}^2+\frac{1}{2}\langle f(x)-f(y),x-y\rangle+\frac{1}{4}\norm{f(x)-f(y)}^2\\
&=\textstyle\norm{\frac{1}{2}(x+f(x))-\frac{1}{2}(y+f(y))}^2
\end{align*}
and
\begin{align*} 
b(x,y) &:=\textstyle\frac{1}{2}\norm{x-y}^2+\frac{1}{2}\left\langle f(x)-f(y),x-y\right\rangle\\ &=\textstyle\left\langle \frac{1}{2}(x+f(x))-\frac{1}{2}(y+f(y)), x-y\right\rangle.
\end{align*}
Then $x\mapsto \frac{1}{2}(x+f(x))$ is firmly non-expansive if and only if $a(x,y)\leq b(x,y)$ for all $x,y\in S$. Moreover, for given $x,y\in S$, the inequality $a(x,y) \leq b(x,y)$ holds if and only if  $\norm{f(x)-f(y)}\leq \norm{x-y}$.
\end{proof}

\subsection{Minkowski sum}\label{sec:Minkowski}
Let $A$ and $B$ be subsets of $R^n$. We denote the (Minkowski) sum of $A$ and $B$ by 
$$A+B=\{a+b:a\in A,\ b\in B\}.$$ If both $A$ and $B$ are closed, then $A+B$ is not necessarily closed, as the example $$A=\{0\}\times R, \quad B=\big\{(x,y)\in R^2: xy\geq 1,\ x\geq 0\big\}$$ shows.
The following fact is used in Section~\ref{sec:Helly applications}.

\begin{lemma}\label{lem:A+B closed}
Let $A,B\subseteq R^n$ be definable, and suppose $A$ is closed, and $B$ is closed and bounded. Then $A+B$ is closed.
\end{lemma}
\begin{proof}
Let $z\in \cl(A+B)$. Then for every $\varepsilon>0$ the definable closed set
$$C_\varepsilon := \big\{ (a,b)\in A\times B : ||a+b-z||\leq\varepsilon\big\}$$
is non-empty. Note that each $C_\varepsilon$ is bounded: if $(a,b)\in C_\varepsilon$ then $||a||\leq ||a+b-z||+||b-z||\leq \varepsilon+\varrho+||z||$ where $\varrho>0$ is such that $B\subseteq B_{\varrho}(0)$. Hence by Lemma~\ref{lem:monotone} we have $\bigcap_{\varepsilon>0} C_\varepsilon\neq\emptyset$, showing that $z\in A+B$.
\end{proof}

\section{Basic Properties of Convex Sets}\label{sec:convex}

\noindent
In this section, $\mathfrak R$ is an expansion of an ordered field $R$.
Recall: $A\subseteq R^n$ is {\bf convex}\/ if for all $x,y\in A$ we have $[x,y]\subseteq A$. 
Here and below, for $x,y\in R^n$ we write
$$[x,y] = \big\{\lambda x+(1-\lambda)y:0\leq\lambda\leq 1\big\}$$
for the line segment in $R^n$ connecting $x$ and $y$. (We also use analogous notation for the half-open line segments $(x,y]$ and $[x,y)$.)
If $A$, $B$ are convex, then so are $A+B$ and $\lambda A=\{\lambda a:a\in A\}$, where $\lambda\in R$.

\subsection{Theorems of Carath\'eodory, Radon, and Helly}
The  intersection of an arbitrary family of convex subsets of $R^n$ is convex. In particular, the intersection of all convex subsets of $R^n$ which contain a given set $A\subseteq R^n$ is a convex set containing $A$, called the {\bf convex hull}\/ $\conv(A)$ of $A$. As in the case $R=\R$ (cf., e.g., \cite[Theorem~2.2.2]{Webster}),  one shows that $\conv(A)$ is the set of {\bf convex combinations}\/ of elements of $A$, that is, the set of $x\in R^n$ for which there are $x_1,\dots,x_k\in R^n$ and $\lambda_1,\dots,\lambda_k\in R^{\geq 0}$ such that $x=\sum_i \lambda_i x_i$ and $\sum_i\lambda_i=1$. In fact, only convex combinations of $n+1$ elements of $A$ need to be considered:

\begin{lemma}[Carath\'eodory's Theorem]
Let $A$ be a subset of $R^n$, and let $x\in\conv(A)$. Then $x$ is a convex combination of affinely independent points in $A$. In particular, $x$ is a convex combination of at most $n+1$ points in $A$.
\end{lemma}

This is also shown just as for $R=\R$, cf.~\cite[Theorem~2.2.4]{Webster}. 
We record some consequences of this lemma. First, an obvious yet important observation:

\begin{corollary}
The convex hull of every definable subset of $R^n$ is definable. 
\end{corollary}

Clearly the convex hull of a bounded subset of $R^n$ is bounded.
The union of a line and a point not on it shows that the convex hull of a closed definable set need not be closed. However, we have:

\begin{corollary}\label{cor:conv of cbd is cbd}
Let $A\subseteq R^n$.
Then $\conv(\cl(A))\subseteq\cl(\conv(A))$.
Moreover, if $\mathfrak R$ is definably complete and $A$ is definable and bounded, then $\conv(\cl(A))=\cl(\conv(A))$; in particular, the convex hull of every closed and bounded definable set is closed and bounded.
\end{corollary}
\begin{proof}
It is easy to see that the closure of a convex set is convex; this yields $\conv(\cl(A))\subseteq\cl(\conv(A))$. Now suppose $\mathfrak R$ is definably complete and $A$ is definable and bounded. Then the subset
$$C := \left\{ (\lambda_1,\dots,\lambda_{n+1},x_1,\dots,x_{n+1}) :  
\lambda_i\geq 0,\ x_i\in\cl(A),\ \sum_{i=1}^{n+1} \lambda_i=1
\right\}$$
of $R^{2(n+1)}$ is definable, closed, and bounded. Hence by Proposition~\ref{prop:cbd} its image under the definable continuous map
$$(\lambda_1,\dots,\lambda_{n+1},x_1,\dots,x_{n+1})\mapsto \sum_{i=1}^{n+1} \lambda_i x_i\in R^n$$
is also closed and bounded. By Carath\'eodory's Theorem, this image is equal to $\conv(\cl(A))$. Thus $\cl(\conv(A))\subseteq\cl(\conv(\cl(A))=\conv(\cl(A))$.
\end{proof}

The next fact is also shown as in the case $R=\R$; cf.~\cite[Theorem~2.2.5]{Webster}.

\begin{lemma}[Radon's Lemma]
Each finite set of affinely dependent points in $R^n$ is a union of two disjoint sets whose convex hulls have a common point.
\end{lemma}

As for $R=\R$, Radon's Lemma implies Theorem~\ref{thm:Helly} in the case of a finite family of convex sets; see \cite[Theorem~7.1.1]{Webster} for a proof. Given a family $\mathcal F=\{F_i\}_{i\in I}$ of sets, we say that $\mathcal F$ has the {\bf $n$-intersection property} if $F_{i_1}\cap\cdots\cap F_{i_n}\neq\emptyset$ for all $i_1,\dots,i_n\in I$, and we say that $\mathcal F$ has the {\bf finite intersection property} if $\mathcal F$ has the $n$-intersection property for some $n$.

\begin{corollary}[Helly's Theorem for finite families] \label{cor:Helly finite}
Let $A_1,\dots,A_k\subseteq R^n$ be convex. If $\{A_i\}_{i=1,\dots,k}$ has the $(n+1)$-intersection property, then $A_1\cap\cdots\cap A_k\neq\emptyset$.
\end{corollary}

The following consequence for arbitrary families of convex sets is perhaps  well-known, but we could not locate it in the literature:

\begin{corollary}\label{cor:Helly finite, 2}
Let $\mathcal C=\{C_a\}_{a\in A}$ be a family of convex subsets of $R^n$, and suppose $p_1,\dots,p_k\in R^n$ have the property that for all $a_1,\dots,a_{n+1}\in A$ there is some $i\in\{1,\dots,k\}$ such that $p_i\in C_{a_1}\cap\cdots\cap C_{a_{n+1}}$. Then $\bigcap\mathcal C\neq\emptyset$.
\end{corollary}
\begin{proof}
Let $P=\{p_1,\dots,p_k\}$.
Then $\mathcal P=\{ \conv(C_a\cap P) \}_{a\in A}$ is a family of convex subsets of $R^n$ with only finitely many distinct members, and by assumption, $\mathcal P$ has the $(n+1)$-intersection property. Hence $\emptyset\neq\bigcap\mathcal P\subseteq \bigcap\mathcal C$ by Corollary~\ref{cor:Helly finite}.
\end{proof}

\subsection{Convex functions}
Let $f\colon S\to R_{\pm\infty}$, where $S\subseteq R^n$.  
The {\bf epigraph of $f$} is the set 
\[
\epi(f)=\big\{ (x,t)\in R^n\times R: x\in S,\ t\geq f(x)\big\}.
\]
We say that $f$ is {\bf convex} if $\epi(f)$ is a convex subset of $R^{n+1}$, and we say that $f$ is {\bf concave} if $-f$ is convex.
Clearly if $f$ is convex, then its {\bf domain}
$$\dom(f)=\{x\in S:f(x)<+\infty\}$$ is a convex subset of $R^n$, since $\dom(f)=\pi(\epi(f))$ where $\pi\colon R^{n+1}\to R^n$ is the natural projection onto the first $n$ coordinates.
We say that $f$ is {\bf proper} if $\epi(f)$ is non-empty and contains no vertical lines, i.e., $f(x)<+\infty$ for some $x\in S$ and $f(x)>-\infty$ for all $x\in S$. Otherwise, $f$ is called {\bf improper}.

\begin{example}
Suppose $S$ is a convex subset of $R^n$ and $f(x)>-\infty$ for all $x\in S$. Then $f$ is convex if and only if for all $x,y\in S$ and $\lambda\in [0,1]$ we have
$$f(\lambda x+ (1-\lambda) y) \leq \lambda f(x) + (1-\lambda) f(y)\qquad\text{for all $\lambda\in [0,1]$,}$$
where this inequality is interpreted in $R_\infty$.
If $f$ is finite and convex, then extending $f$ by setting $f(x):=+\infty$ for $x\in R^n\setminus S$ yields a convex function $R^n\to R_\infty$ (and every proper convex function $R^n\to R_\infty$ arises in this way from the restriction to its domain). For example, the constant function $0$ on $S$ extends to a convex function $\delta_S\colon R^n\to R_\infty$ with $\delta_S|(R^n\setminus S)\equiv +\infty$, called the
{\bf indicator function of  $S$.}
\end{example}

We say that $f$ is definable if the restriction of $f$ to the set $f^{-1}(R)$ of points at which $f$ is finite  is definable (as function $f^{-1}(R)\to R$). Similarly, a family $\{f_a\}_{a\in A}$ of functions $f_a\colon S_a\to R_{\pm\infty}$ (where $A\subseteq R^m$ and $S_a\subseteq R^n$ for every $a\in A$) is called definable if the family $\{f_a|f_a^{-1}(R)\}_{a\in A}$ is definable. 



\subsection{Constructing convex functions}

{\it Throughout the rest of this section, we assume that $\mathfrak R$ is definably complete.}\/
The following lemma (which is easy to verify) shows in particular that the pointwise supremum of a definable family of convex functions is convex:

\begin{lemma}\label{lem:sup of convex functions}
Let $f\colon R^n\to R_{\pm\infty}$, and let $\{f_a\}_{a\in A}$ be a definable family of functions $f_a\colon R^n\to R_{\pm\infty}$. Then
$$f=\sup_{a\in A} f_a\quad\Longleftrightarrow\quad \epi(f)=\bigcap_{a\in A}\epi(f_a).$$ 
\end{lemma}

The next lemma is also easily proved; it allows the construction of convex functions from fibers of definable convex sets:

\begin{lemma}
Let $C$ be a convex definable subset of $R^{n+1}$. Then  $f\colon R^n\to R_{\pm\infty}$ defined by $f(x) = \inf C_x$ is convex with domain $\pi(C)$, where $\pi\colon R^{n+1}\to R^n$ is the projection onto the first $n$ coordinates.
\end{lemma}

Let now $f\colon R^n\to R_{\pm\infty}$ be definable, and let $A\colon R^n\to R^m$ be $R$-linear. We denote the definable function
$$x\mapsto \inf\big\{f(y):A(y)=x\big\}\colon R^m\to R_{\pm\infty}$$
by $Af$. Applying the lemma above to  $C=(A\times\id)(\epi(f))$ yields:

\begin{lemma}
Suppose $f$ is convex. Then  $Af$ is convex with domain $A(\dom(f))$.
\end{lemma}

Let  $f,g\colon R^n\rightarrow R_\infty$ be definable.
The (infimal) {\bf convolution} $f\boxempty g\colon R^n\rightarrow R_{\pm\infty}$ of $f$ and $g$ is defined by 
\[
(f\boxempty g)(x):=\inf_{y\in R^n} \big(f(y)+g(x-y)\big)\qquad\text{for $x\in R^n$.}
\]
By the previous lemma,
if $f$ and $g$ are convex, then $f\boxempty g$ is convex, with domain $\dom(f)+\dom(g)$. (However, if $f$ and $g$ are proper, $f\boxempty g$ may fail to be proper, as the example $f=\id_R$, $g=-\id_R$ shows.)
Note that for every $x\in R^n$,
$$(f\boxempty g)(x) = \inf\big\{s+t: (y,s)\in\epi(f),\ (z,t)\in\epi(g),\ y+z=x\big\}$$
and hence
\begin{equation}\label{eq:epi of convolution}
\epi(f)+\epi(g) \subseteq \epi(f\boxempty g)\subseteq \cl(\epi(f)+\epi(g)).
\end{equation}

\subsection{The distance function and the metric projection}
After this digression on convex functions, we return to the study of convex sets.
{\it For the rest of this section we fix a non-empty convex closed definable subset $C$ of $R^n$.}
Recall that the distance from $x\in R^n$ to $C$ is defined by
$$d(x,C) = \inf\big\{|| x-c ||:c\in C\big\}.$$
We have:

\begin{lemma}\label{lem:convex}
The function
$x\mapsto d(x,C)\colon R^n\to R$
is convex.
\end{lemma}
\begin{proof}
The function $d(\,\cdot\, , C)$ may be expressed as the convolution of the Euclidean norm and the indicator function of $C$.
\end{proof}


\begin{lemma} For every $x\in R^n$, there is a unique element of $C$ of smallest distance to $x$.
\end{lemma}
\begin{proof}
We have existence by Corollary~\ref{cor:nearest point}. For uniqueness,
let $x\in R^n$, and suppose $y_1,y_2\in C$ are both nearest points of $C$ to $x$. Then $z:=\frac{1}{2}(y_1+y_2)\in C$ and $||x-z||<||x-y_1||$ except if $y_1=y_2$.
\end{proof}

Given $x\in R^n$, we denote the unique nearest point to $x$ in $C$ by $p(x,C)$. The map $x\mapsto p(x,C)\colon R^n\to C$ is called the (metric) {\bf projection}\/ of $C$. Note that
$$d(x,C)=||x-p(x,C)||=\min \big\{|| x-y ||:y\in C\big\}.$$

\begin{lemma}\label{lem:pyth}
For all $x\in R^n$ and $z\in C$ we have
$$\big\langle x-p(x,C), z-p(x,C)\big\rangle \leq 0.$$
\end{lemma}
\begin{proof}
Let $x\in R^n$, $z\in C$, and put $p:=p(x,C)$. For $0<\lambda\leq 1$ set
$z_\lambda := \lambda\,z+(1-\lambda)\,p$.
Then $z_\lambda\in C$ and hence
$$\norm{x-p}^2 \leq \norm{x-z_\lambda}^2 = \norm{(x-p)+\lambda(z-p)}^2.$$
Subtracting $\norm{x-p}^2$ yields
$0\leq \lambda^2\norm{z-p}^2-2\lambda\langle x-p,z-p\rangle$.
Dividing by $\lambda$ and taking $\lambda\to 0$ yields the lemma.
\end{proof}

\begin{corollary}\label{cor:non-expansive}
The projection $p(\,\cdot\, , C)$ is firmly non-expansive.
\end{corollary}
\begin{proof}
Let $x,y\in R^n$; we need to show that $$\langle x-y,p(x,C)-p(y,C)\rangle\geq \norm{p(x,C)-p(y,C)}^2.$$
To see this apply Lemma~\ref{lem:pyth} to $(x,p(y,C))$ and $(y,p(x,C))$ in place of $(x,z)$, respectively, and add the resulting inequalities.
\end{proof}

The following lemma is used in the proof of Theorem~\ref{thm:Helly} in the next section:

\begin{lemma}\label{lem:lem1}
Suppose $C$ is bounded, and let $x\in R^n\setminus C$, $p=p(x,C)$, and $z\in (x,p]$. Then there is a $\delta>0$ such that
$||z-c||\leq ||x-c||-\delta$
for every $c\in C$.
\end{lemma}
\begin{proof}
Suppose not. Then for every $\delta>0$ the set
$$C_\delta := \big\{ c\in C : ||z-c||\geq ||x-c||-\delta\big\}$$
is non-empty, and so we have a decreasing  definable family $\{C_\delta\}_{\delta>0}$ of closed and bounded non-empty sets. Hence by Lemma~\ref{lem:monotone} there is some $c\in C$ with $||z-c||\geq ||x-c||$. Let $a$ be the point of smallest distance to $c$ on the line through $x$ and $p$. 
Then Pythagoras yields $||a-z||\geq ||a-x||$, a contradiction to $x\neq z$.
\end{proof}

\subsection{Supporting hyperplanes}
Given $p,u\in R^n$, $u\neq 0$, and $\alpha\in R$ we write
$$H_{u,p}=\big\{y\in R^n: \langle y,u\rangle = \langle p,u\rangle \big\}$$
for the hyperplane in $R^n$ through $p$ orthogonal to $u$. Note that if $\langle p,u\rangle=\langle p',u\rangle$ then $H_{u,p}=H_{u,p'}$, and we sometimes write $H_{u,\alpha}$ for $H_{u,p}$, where $\alpha=\langle p,u\rangle$.
Given a hyperplane  $H=H_{u,\alpha}$, we write
$$H^+ = \big\{y\in R^n: \langle y,u\rangle \geq \alpha \big\}, \quad
  H^- = \big\{y\in R^n: \langle y,u\rangle \leq \alpha \big\}$$
for the two closed halfspaces bounded by $H$. 

Let $S\subseteq R^n$. Given a hyperplane $H=H_{u,\alpha}$ and a point $x\in R^n$, we say that {\bf $H$ supports $S$ at $x$} if $x\in S\cap H$ and $S\subseteq H^+$ or $S\subseteq H^-$. (In this case necessarily $x\in\bd(S)$.)

\begin{lemma}\label{lem:supporting hyperplane}
Let $x\in R^n\setminus C$.
The hyperplane $H=H_{u,p}$ through $p=p(x,C)$ orthogonal to $u=x-p$ supports $C$ at $p$, and $C$ is contained in the halfspace $H^-$ bounded by $H$ which does not contain $x$.
\end{lemma}
\begin{proof}
By Lemma~\ref{lem:pyth} and since $x\neq p$, for every $y\in C$ we have
$\langle y,x-p\rangle \leq \langle x-p,p\rangle < \langle x, x-p\rangle$,
and this yields the lemma.
\end{proof}

In particular, the previous lemma implies that if
$C\neq R^n$, then $C$ is the intersection of all closed halfspaces which contain $C$. 

\begin{corollary}\label{cor:support}
For each $p\in\bd(C)$ there exists some $y\in\bd(B_1(p))$ with $p=p(y,C)$. Hence for every $p\in\bd(C)$ there is a hyperplane that supports $C$ at $p$.
\end{corollary}
\begin{proof}
Let $p\in\bd(C)$. For every $\varepsilon$ with $0<\varepsilon<1$ there is some $x\in \ol{B}_\varepsilon(p)\setminus C$; then $||p-p(x,C)||\leq ||p-x||<\varepsilon$ by Corollary~\ref{cor:non-expansive}, and since the distinct points $x$ and $p(x,C)$ are in $B_1(p)$, there is a (unique) $y$ on the sphere $S:=\bd(B_1(p))$ with $x\in [p(x,C),y]$. By Lemma~\ref{lem:supporting hyperplane} we have $p(x,C)=p(y,C)$. This means that for every $\varepsilon>0$ the definable set
$$S_\varepsilon := \big\{ y\in S: \ol{B}_\varepsilon(p) \cap [p(y,C),y]\neq\emptyset \big\}$$
is non-empty. It is easily verified that each $S_\varepsilon$ is closed and bounded. By Lemma~\ref{lem:monotone} take $y\in\bigcap_{\varepsilon>0} S_\varepsilon$. Then for every $\varepsilon>0$ there is some $x\in [p(y,C),y]$ with $||x-p||\leq\varepsilon$; hence $||p(y,C)-p||=||p(x,C)-p||\leq ||x-p||\leq\varepsilon$. Thus $p=p(y,C)$.
\end{proof}

\begin{remark}
The existence of a supporting hyperplane through every boundary point characterizes convex sets among definable closed subsets of $R^n$ with non-empty interior; this can be shown as in the case $R=\R$, see, e.g., \cite[Theorem~1.3.3]{Schneider}.
\end{remark}

\subsection{Separating hyperplanes}
Let $A,B\subseteq R^n$ and let $H=H_{u,\alpha}$ be a hyperplane. We say that $H$ {\bf separates $A$ and $B$} if $A\subseteq H^-$ and $B\subseteq H^+$, or vice versa. If there is a hyperplane separating $A$ and $B$, we say that $A$ and $B$ can be separated.

\begin{proposition}\label{prop:separation}
Let $S\subseteq R^n$ be definable, non-empty, and convex, and let $x\in R^n\setminus S$.
Suppose $S$ is closed, or $S$ is open. Then $S$ and $\{x\}$ can be separated.\end{proposition}

To see this we use:

\begin{lemma}
Let $A\subseteq R^n$ be convex with non-empty interior. Then $\interior(\cl(A))=\interior(A)$ and hence $\bd(A)=\bd(\cl(A))$.
\end{lemma}
\begin{proof}
The inclusion $\interior(A)\subseteq\interior(\cl(A))$ is trivial. Conversely, let $z\in\interior(\cl(A))$. Take an arbitrary $x\in\interior(A)$. Then there exists $y\in\cl(A)$ such that $z\in [x,y)$.
As in the case $R=\R$ (cf.~\cite[Lemma~1.1.8]{Schneider}) one shows that this implies $z\in\interior(A)$.
\end{proof}

\begin{proof}[Proof \textup{(}Proposition~\ref{prop:separation}\textup{)}]
Suppose first that $S$ is closed, and set $p=p(x,S)$, $u=x-p$. Then the hyperplane which is parallel to the supporting hyperplane $H_{u,p}$ of $S$ at $p$ and passes through $(p+x)/2$ separates $S$ and $\{x\}$. If $S$ is not closed and $x\notin\cl(S)$, then every hyperplane separating $\cl(S)$ and $\{x\}$ also separates $S$ and $\{x\}$. If $S$ is open and $x\in\cl(S)$, then $x\in\bd(\cl(S))$, so by Corollary~\ref{cor:support} there is a supporting hyperplane $H$ to $\cl(S)$ through $x$, and $H$ separates $S$ and $\{x\}$.
\end{proof}

\begin{remark}
Proposition~\ref{prop:separation} fails if the requirement that $S$ be definable is dropped.
(See \cite{Robson} for what can be salvaged in this case by employing the real spectrum.)
\end{remark}

We obtain a definable version of a special case of the separation theorem for convex sets \cite[Theorem~2.4.10]{Webster}:

\begin{corollary}\label{cor:separation}
Let $A,B\subseteq R^n$ be definable non-empty convex sets with $A\cap B=\emptyset$. If $A$ is open, or if $A$ is closed and $B$ is closed and bounded, then $A$ and $B$ can be separated.
\end{corollary}
\begin{proof}
The convex set $S:=A-B$ does not contain the origin $0$ of $R^n$. If $A$ is open, then so is $S$, and if $A$ is closed and $B$ is closed and bounded, then $S$ is closed (Lemma~\ref{lem:A+B closed}). Hence $S$ and $\{0\}$ can be separated by Proposition~\ref{prop:separation}. It is easy to see that this yields that $A$ and $B$ can be separated.
\end{proof}

\section{Proof of Theorem~\ref{thm:Helly}}\label{sec:Helly proof}

\noindent
Suppose that $\mathfrak R$ is a definably complete expansion of an ordered field, and let $\mathcal C=\{C_a\}_{a\in A}$ be a definable family of closed bounded convex subsets of $R^n$, with $A\neq\emptyset$. Assume $\mathcal C$ has the $(n+1)$-intersection property; we need to show $\bigcap\mathcal C\neq\emptyset$. Fix an arbitrary $a_0\in A$. By Helly's Theorem for finite families (Corollary~\ref{cor:Helly finite}), the definable family $\mathcal C'=\{C_a\cap C_{a_0}\}_{a\in A}$ of closed bounded convex subsets of $R^n$ also has the $(n+1)$-intersection property. Hence, after replacing $\mathcal C$ by $\mathcal C'$ if necessary, we may assume that $\bigcup_{a\in A} C_a$ is bounded. In particular, for each $x\in R^n$, the set of distances $d(x,C_a)$ (where $a$ ranges over $A$) is bounded from above, and we obtain a definable function $d\colon R^n\to R$ given by
$$d(x) := \sup_{a\in A} d(x,C_a).$$
The function $d$ is convex and non-expansive. (Lemmas~\ref{lem:inf and sup of Lipschitz functions} and \ref{lem:sup of convex functions}, and Corollary~\ref{cor:distance is non-expansive}.)
In particular, for $\varrho>0$ such that $B_{\varrho/2}(0)\supseteq \bigcup_{a\in A} C_a$, the 
restriction of $d$ to $\ol{B}_\varrho(0)$ has a minimum. (Corollary~\ref{cor:minmax}.) This minimum must be attained in $B_\varrho(0)$, and is indeed a global minimum of $d$. 
Let $x_0\in R^n$ such that $d(x_0)=\min_{x\in R^n} d(x)$. If $d(x_0)=0$ then $x_0\in \bigcap_{a\in A} C_a$, and we are done. So assume $d(x_0)>0$. We obtain a definable map 
$$a\mapsto x_a:=p(x_0,C_a)\colon A\to \bigcup_a C_a.$$
We have $||x_0-x_a||=d(x_0,C_a)$ for each $a\in A$. 
Let $\varepsilon>0$ be given. The definable set 
$$A_\varepsilon := \big\{ a\in A: d(x_0)-\varepsilon \leq d(x_0,C_a)\big\}$$
is non-empty. We let $H$ be the image of $A_\varepsilon$ under $a\mapsto x_a$, and put $C:=\cl(\conv(H))$. (There is no reason to believe that $\conv(H)$ is closed, unless, for example, $A_\varepsilon$ is finite.) 

\begin{claim} $x_0\in C$.
\end{claim}
\begin{proof}
Suppose for a contradiction that $x_0\notin C$.
Let $p=p(x_0,C)$, and let $z\in [x_0,p]$ such that $||x_0-z||=\varepsilon/2$. We show that $d(z)<d(x_0)$; this will contradict the minimality of $d(x_0)$. If $a\notin A_\varepsilon$, then $d(x_0)-\varepsilon > d(x_0,C_a)$, and since $d(\,\cdot\, , C_a)$ is non-expansive, we have
$$d(z,C_a) - d(x_0,C_a) \leq ||z-x_0||=\varepsilon/2$$
and hence
$$d(z,C_a) \leq \varepsilon/2 + (d(x_0)-\varepsilon) = d(x_0)-\varepsilon/2.$$
Let $\delta>0$ be as in Lemma~\ref{lem:lem1} applied to $x=x_0$. Then for all $a\in A_{\varepsilon}$ we have
$$d(z,C_a)\leq ||z-x_a|| \leq ||x_0-x_a||-\delta=d(x_0,C_a)-\delta \leq d(x_0)-\delta.$$
Hence $d(z)\leq d(x_0)-\min(\varepsilon/2,\delta)$.
\end{proof}

By the claim and Carath\'eodory's theorem there are elements $a_1,\dots,a_{n+1}\in A_\varepsilon$ and non-negative $\lambda_1,\dots,\lambda_{n+1}\in R$ with $\sum_i \lambda_i=1$ and
$||x_0 - \sum_i \lambda_i x_{a_i}||<\varepsilon^2$.
By Lemma~\ref{lem:supporting hyperplane}, for $y\in C_{a_i}$ we have
$$\langle y-x_0, x_{a_i} - x_0 \rangle\geq ||x_{a_i}-x_0||^2, $$
and since $$||x_{a_i}-x_0||=d(x_0,C_{a_i})\geq d(x_0)-\varepsilon,$$
we obtain
\begin{equation}\label{eq:lower bound}
\langle y-x_0, x_{a_i} - x_0 \rangle\geq (d(x_0)-\varepsilon)^2.
\end{equation}
Take $y$ with $y\in C_{a_i}$ for all $i$. (Such $y$ exists by the assumption of the theorem.) Then using \eqref{eq:lower bound} and the Cauchy-Schwarz Inequality we get
\begin{align*}
(d(x_0)-\varepsilon)^2 &\leq \sum_i \lambda_i \langle y-x_0, x_{a_i} - x_0 \rangle \\
&= \left\langle y-x_0, \sum_i \lambda_i x_{a_i}-x_0\right\rangle \\
&\leq  ||y-x_0||\cdot \varepsilon^2 \leq r\cdot \varepsilon^2,
\end{align*}
where $r>0$ is such that $\conv(\bigcup_{a\in A} C_a)\subseteq \ol{B}_r(x_0)$. Hence $\left(\frac{d(x_0)}{\varepsilon}-1\right)^2\leq r$,
and this is a contradiction for sufficiently small $\varepsilon>0$. \qed

\begin{remark}
The proof of Theorem~\ref{thm:Helly} given above exploits a certain duality between the intersection properties of convex sets and the representation of elements in the convex hull. After a first version of this manuscript was completed, we became aware of Sandgren's proof of Helly's Theorem \cite{Sandgren} (in the exposition of Valentine \cite{Valentine}) in which this duality is made more explicit. This proof may probably be adapted to give another proof of Theorem~\ref{thm:Helly} above. 
\end{remark}

\subsection{Applications}\label{sec:Helly applications}
In this subsection we give some applications of Theorem~\ref{thm:Helly}. Throughout we assume that $\mathfrak R$ is a definably complete expansion of an ordered field. By a {\bf translate} of  $A\subseteq R^n$ we mean a set of the form $x+A$, for some $x\in R^n$.
The following generalizes Theorem~\ref{thm:Helly} (which corresponds to the case  where $K$ is a singleton):

\begin{corollary}
Let $\mathcal C=\{C_a\}_{a\in A}$ be a definable family of closed bounded convex subsets of $R^n$, and let $K\subseteq R^n$ be definable, closed, bounded and convex. If any $n+1$ elements of $\mathcal C$ intersect some translate of $K$ non-trivially, then there is a translate of $K$ intersecting every element of $\mathcal C$ non-trivially.
\end{corollary}
\begin{proof}
Recall  Lem\-ma~\ref{lem:A+B closed} and
apply Theorem~\ref{thm:Helly} to the family $\{K-C_a\}_{a\in A}$.
\end{proof}

Combining Helly's Theorem for finite families (Corollary~\ref{cor:Helly finite}) with Theorem~\ref{thm:Helly} yields another slight variant:

\begin{corollary}\label{cor:Helly only one bounded}
Suppose $\mathcal C$ is a definable family of closed convex subsets of $R^n$, each $n+1$ of which intersect non-trivially, and assume some intersection $C$ of finitely many members of $\mathcal C$ is bounded. Then $\bigcap\mathcal C\neq\emptyset$.
\end{corollary}

By taking complements, Theorem~\ref{thm:Helly} about intersections of closed sets immediately gives rise to a result about coverings by open sets:

\begin{corollary}
Let $\mathcal F=\{F_a\}_{a\in A}$ be a definable family of open subsets of $R^n$ with the property that for every $a\in A$, the complement $R^n\setminus F_a$ is convex. Let $C$ be a closed bounded convex definable subset of $R^n$ with $C\subseteq\bigcup\mathcal F$. Then there are $n+1$ members $F_{a_1},\dots, F_{a_{n+1}}$ of $\mathcal F$ with $C\subseteq F_{a_1}\cup\cdots\cup F_{a_{n+1}}$.
\end{corollary}

The hypotheses on $\mathcal F$ are satisfied, e.g., by the family of open halfspaces in $R^n$. 

\begin{corollary}[Jung's Theorem]
Let $A$ be a definable subset of $R^n$ of diameter at most $1$ \textup{(}i.e., $||a-b||\leq 1$ for all $a,b\in A$\textup{)}. Then there is a closed ball of radius $\varrho=\sqrt{n/(2(n+1))}$ containing $A$.
\end{corollary}
\begin{proof}
If $A$ has at most $n+1$ elements, this may be shown as for $R=\mathbb R$, cf.~\cite[Theorem~7.1.6]{Webster}. Hence for arbitrary $A$, by Theorem~\ref{thm:Helly} there is some $x\in \bigcap_{a\in A} \ol{B}_\varrho(a)$, and then $A\subseteq \ol{B}_\varrho(x)$.
\end{proof}

Given a subset $A$ of $R^n$, a family $\{C_a\}_{a\in A}$ of subsets of $A$ is called a {\bf Knaster-Kuratowski-Mazurkiewicz family} (KKM family for short) if for every finite subset $F$ of $A$,
$$\conv(F) \subseteq \bigcup_{a\in F} C_a.$$
The KKM Theorem (see \cite{GD,GL}) states that if $A$ is a non-empty compact convex subset of $\R^n$, then every KKM family consisting of closed subsets of $A$ has a non-empty intersection. From Theorem~\ref{thm:Helly} we obtain:

\begin{corollary}[KKM Theorem for definable families of convex sets]
Let $A$ be a closed and bounded non-empty subset of $R^n$ and let $\mathcal C=\{C_a\}_{a\in A}$ be a KKM family where each $C_a$ is closed and convex. Then $\bigcap\mathcal C\neq\emptyset$.
\end{corollary}
\begin{proof}
The argument in the proof of \cite[Th\'eor\`eme~1]{GL} for the case $R=\R$ shows that $\mathcal C$ has the finite intersection property. Hence $\bigcap\mathcal C\neq\emptyset$ by Theorem~\ref{thm:Helly}.
\end{proof}

The KKM Theorem for convex sets has numerous consequences (minimax theorems etc.), whose proofs go through for definable objects; cf.~\cite{GD, GL}.
Other applications of Helly's Theorem, some of which may also be transferred into the present context, can be found in \cite{DGK, Eckhoff}. 
Our last application of Theorem~\ref{thm:Helly} is used in the proof of Theorem~\ref{thm:Kirszbraun} in the next sections:
\begin{corollary}\label{cor:extend non-expansive}
Let $f\colon A\rightarrow R^n$, $A\subseteq R^n$, be a definable non-expansive map, and let $x\in R^n\setminus A$.
Then $f$ extends to a non-expansive map $A\cup\{x\}\to R^n$. 
\end{corollary}
\begin{proof}
We have to show that  the set
\[
B:=\bigcap_{a\in A}\big\{ y\in R^n: \norm{y-f(a)}\leq \norm{x-a}\bigr\}
\]
is non-empty, because if $y\in B$, then we obtain an extension of $f$ to a non-expansive map $A\cup\{x\}\to R^n$ by $x\mapsto y$.

\begin{claim}
Let $x_1,\dots,x_{k}\in R^m$ and $y_1,\dots, y_{k}\in R^n$ for which the inequalities
\[
\norm{y_i-y_j}\leq \norm{x_i-x_j} \qquad (1\leq i,j\leq k)
\]
hold, and let $r_1,\dots,r_k\in R^{>0}$. If
$$\ol{B}_{r_1}(x_1)\cap\cdots\cap\ol{B}_{r_k}(x_k)\neq\emptyset,$$
then 
$$\ol{B}_{r_1}(y_1)\cap\cdots\cap\ol{B}_{r_k}(y_k)\neq\emptyset.$$
\end{claim}
(To see this, repeat the proof for the case $R=\R$ given in \cite[Lemma 2.7]{Heinonen}, or use the fact that the claim can be expressed as a sentence in the language of ordered rings, and apply Tarski's Transfer Principle and loc.~cit. A stronger version of the claim for $R=\R$ can be found in \cite{Gromov}.)
\medskip

Consider the definable family $\mathcal B=(B_a)_{a\in A}$ of closed balls in $R^n$ given by 
\[
B_{a}:=\ol{B}_{||x-a||}\big(f(a)\big)=\big\{ y\in R^n: \norm{y-f(a)}\leq \norm{x-a}\big\}.
\]
Let $a_1,\dots,a_{n+1}\in A$. Then $y_i:=f(a_i)$ and $x_i:=x-a_i$ satisfy the conditions of the claim, hence $B_{a_1}\cap \dots\cap B_{a_{n+1}}\neq\emptyset$. 
Thus by Theorem \ref{thm:Helly} we have $B=\bigcap\mathcal B\neq \emptyset$.
\end{proof}

\begin{remark}
In the context of the previous corollary, suppose that $f$ is firmly non-expansive. Then there exists an extension of $f$ to a firmly non-expansive map $A\cup\{x\}\to R^n$, by the corollary and Proposition~\ref{prop44}.
\end{remark}

\subsection{A related result}\label{sec:related results}

Let $S$ be a set and let $\mathcal F$ be a family of subsets of $S$.
A subset $T$ of $S$ is called a {\bf transversal}\/ of $\mathcal F$ if every member of $\mathcal F$ intersects $T$ non-trivially.
The following was shown by Peterzil and Pillay \cite{PP}, as an application of a result implicit in work of Dolich \cite{Dolich}:


\begin{theorem}\label{thm:Dolich}
Let $\mathfrak R$ be an o-minimal structure with definable choice function, and let $\mathcal F=\{F_a\}_{a\in A}$ be a definable family of closed and bounded subsets of $R^n$ parametrized by a subset $A$ of $R^m$. If $\mathcal F$ has the $N(m,n)$-intersection property where 
$$N(m,n) = (1+2^m)\cdot (1+2^{2^m}) \cdots \qquad\text{\textup{(}$n$ factors\textup{)}},$$
then $\mathcal F$ has a finite transversal.
\end{theorem}

This theorem gives rise to another proof of Theorem~\ref{thm:Helly}, kindly communicated to us by S.~Star\-chen\-ko,  in the case where $\mathfrak R$ is an o-minimal expansion of an ordered field.
Suppose $\mathfrak R$ is such an expansion, and let $\mathcal C=\{C_a\}_{a\in A}$ be a definable family of closed bounded convex subsets of $R^n$, with $A\neq\emptyset$, having the $(n+1)$-intersection property.
By Helly's theorem for finite families (Corollary~\ref{cor:Helly finite}), the (definable) family whose members are the intersections of $n+1$ members of $\mathcal C$ has the finite intersection property, and hence has a finite transversal by Theorem~\ref{thm:Dolich}. That is, there are $p_1,\dots,p_k\in R^n$ such that for all $a_1,\dots,a_{n+1}\in A$ we have $p_i\in C_{a_1}\cap\cdots\cap C_{a_{n+1}}$ for some $i$. Now Corollary~\ref{cor:Helly finite, 2} yields $\bigcap\mathcal C\neq\emptyset$.  \qed

\medskip

We finish with an example to show that the natural analogue of the Heine-Borel Theorem fails in the definable category:

\begin{example}
Suppose $\mathfrak R$ is a non-archimedean real closed field, and let $\varepsilon\in R$ be a positive infinitesimal. Then the definable family $\mathcal F=\{F_a\}_{a\in A}$ of closed and bounded subsets of $A=[0,1]$ given by $F_a=[0,1]\setminus (a-\varepsilon,a+\varepsilon)$ for $a\in A$ has the finite intersection property, but $\bigcap\mathcal F=\emptyset$. However, any two distinct elements of $A\cap\mathbb Q$ form a transversal of $\mathcal F$. (This is a simplification of an example in \cite{PP}.)
\end{example}

\section{Basic Convex Analysis}\label{sec:convex analysis}
\noindent
In this section we develop a few fundamental results from convex analysis required for the proof of Theorem~\ref{thm:Kirszbraun}. See \cite{BL, HU-L, Rockafellar} for this material in the classical case.

\subsection{Lower semicontinuous functions}
In this subsection we let $f\colon R^n\to R_{\pm\infty}$ be a function. One says that $f$ is lower semicontinuous (l.s.c.) if for each $x\in R^n$ and $\delta>0$, there exists $\varepsilon>0$ such that $f(y)\geq f(x)-\delta$ for all $y\in B_\varepsilon(x)$. 
A continuous function $R^n\to R$ is clearly l.s.c. Lower semicontinuity may be characterized geometrically:

\begin{lemma}
The following are equivalent:
\begin{enumerate}
\item $f$ is l.s.c.;
\item $\epi(f)$ is closed;
\item for every $r\in R$, the sublevel set
$f^{-1}(R^{\leq r})= \big\{ x\in R^n:f(x)\leq r\big\}$
of $f$ is closed.
\end{enumerate}
\end{lemma}
\begin{proof}
Suppose $f$ is l.s.c., and let $(x,r)\in\cl(\epi(f))$. Let $\delta>0$ be given, and choose $\varepsilon$ with $0<\varepsilon\leq\delta$ as in the definition of l.s.c.~above. There exists $(y,t)\in\epi(f)$ with $||x-y||<\varepsilon$ and $||r-t||<\varepsilon$. Hence
$$r+\delta > t \geq f(y) \geq f(x)-\delta.$$
Since this inequality holds for all $\delta>0$, we obtain $r\geq f(x)$, that is, $(x,r)\in\epi(f)$. This shows (1)~$\Rightarrow$~(2). The implication (2)~$\Rightarrow$~(3) follows from the identity
$$f^{-1}(R^{\leq r}) \times \{r\} = \epi(f)\cap (R^{n}\times \{r\}).$$
Suppose all sublevel sets of $f$ are closed, and let $x\in R^n$ and $\delta>0$ be given. Then $x\notin f^{-1}(R^{\leq r})$, where $r:=f(x)-\delta$ if $f(x)<\infty$ and $r:=0$ otherwise. Hence there exists $\varepsilon>0$ such that $y\notin f^{-1}(R^{\leq r})$ for all $y\in B_\varepsilon(x)$. Thus $f$ is l.s.c.
\end{proof}

For proper convex functions $R^n\to R_\infty$, we use {\bf closed} synonymously with l.s.c., and we also declare the constant functions $+\infty$ and $-\infty$ to be closed. Note that if
$f,g\colon R^n\to R_{\infty}$ are closed convex, then so is $\lambda f+\mu g$, for each $\lambda,\mu\in R^{\geq 0}$.

The following proposition is an analogue for definable convex functions of the supporting hyperplane lemma (Lemma~\ref{lem:supporting hyperplane}). It is proved similar to the case $R=\R$, see \cite[Proposition~IV.1.2.8]{HU-L} or \cite[Theorem~12.1]{Rockafellar}. A function $\varphi\colon R^n\to R$ of the form $x\mapsto \langle x,u\rangle-\alpha$ (where $u\in R^n$, $\alpha\in R$) is called {\bf affine.} 
Alternatively, $\varphi$ is affine if and only if $\varphi$ is both convex and concave. The epigraph of an affine function $R^n\to R$ is a closed halfspace in $R^{n+1}$. Below we let $\varphi$ (possibly with subscripts) range over all affine functions $R^n\to R$.

\begin{proposition}\label{prop:closed convex}
Suppose $f$ is definable. The following are equivalent:
\begin{enumerate}
\item $f$ is closed convex;
\item $f=\sup \big\{ \varphi: \varphi\leq f\big\}$;
\item $f=\sup_{a\in A} \varphi_a$ for some definable family of affine functions  $\{\varphi_a\}_{a\in A}$.
\end{enumerate}
\end{proposition}
\begin{proof}
The implication (2)~$\Rightarrow$~(3) is trivial, and (3)~$\Rightarrow$~(1) follows from Lemma~\ref{lem:sup of convex functions}. To show (1)~$\Rightarrow$~(2),  suppose $f$ is closed convex. We may assume that $f$ is proper, so $\epi(f)$ is a proper non-empty closed convex definable subset of $R^{n+1}$. Hence $\epi(f)$ is the intersection of all closed halfspaces containing $\epi(f)$. (Lemma~\ref{lem:supporting hyperplane}.) As in the case $R=\R$ one now shows that only the hyperplanes corresponding to epigraphs of affine functions are required in this intersection; cf.~proof of \cite[Proposition~IV.1.2.8]{HU-L}. (The reference to \cite[Proposition~1.2.1]{HU-L} in that proof is superfluous.)
\end{proof}


\subsection{Conjugates}
Let $f\colon R^n\to R_{\pm\infty}$ be definable.
The (Fenchel) {\bf conjugate of $f$}  is the definable function  $f^*\colon R^n\rightarrow R_{\pm\infty}$ given by
\[
f^*(x^*):=\sup_{x\in R^n}\big(\langle x,x^*\rangle-f(x)\big).
\]
Note that if there is $x_0\in R^n$ with $f(x_0)=-\infty$, then $f^*\equiv +\infty$, whereas if $f\equiv +\infty$ then $f^* \equiv -\infty$.
Clearly if $g\colon R^n\to R_{\pm\infty}$ is another definable function and $f\leq g$, then $f^*\geq g^*$. We summarize further properties of conjugates in the next lemma:
\begin{lemma}
Let $f\colon R^n\to R_{\pm\infty}$ be definable. Then:
\begin{enumerate}
\item The function $f^*$ is closed convex.
\item We have $f^{**}:=(f^*)^*\leq f$, with equality if and only if $f$ is closed convex. 
\item If $f$ is proper closed convex, then $f^*$ is proper, and
$\langle x,x^*\rangle\leq f(x)+f^*(x^*)$ for all $x,x^*\in R^n$. \textup{(}Fenchel-Young Inequality.\textup{)}
\end{enumerate}
\end{lemma}
\begin{proof}
Clearly $f^*$ is closed convex, being the supremum of a definable family of affine functions. This shows (1), and also that $f$ is closed convex if $f^{**}=f$. It is easy to check that $f^{**}\leq f$, with equality if $f$ is an affine function $R^n\to R$.  Hence if $f$ is closed convex, then
for every affine function $\varphi\colon R^n\to R$ with $\varphi\leq f$ we have $\varphi=\varphi^{**}\leq f^{**}\leq f$. Thus $f^{**}=f$  by Proposition~\ref{prop:closed convex}.
This shows (2). Note that (2) implies that if $f$ is proper closed convex, then $f^*$ is proper, since the only improper closed convex functions are $+\infty$ and $-\infty$, which are conjugate to each other.  The Fenchel-Young Inequality is now immediate.
\end{proof}

Given $\lambda\in R^{>0}$ we define 
$$\lambda\ast f\colon R^n\to R_{\pm\infty},\qquad (\lambda\ast f)(x)=\lambda f(x/\lambda) \quad\text{for $x\in R^n$.}$$ 
Note that if $f$ and $g\colon R^n\to R_\infty$ are definable and proper, then $\lambda\ast (f\boxempty g)=(\lambda\ast f)\boxempty (\lambda\ast g)$.
The formulas in the following lemma are useful for computing conjugates. 

\begin{lemma}\label{lem:conjugates}
Let $f,g\colon R^n\to R_{\infty}$ be definable and convex.
\begin{enumerate}
\item For all $\lambda>0$, we have $(\lambda f)^*=\lambda\ast f^*$ and $(\lambda\ast f)^*=\lambda f^*$.
\item Let $A\colon R^n\to R^m$ be $R$-linear. Then $(Af)^*=f^*\circ A^*$, where $A^*\colon R^m\to R^n$ is the adjoint of $A$. In particular, $(f\boxempty g)^* = f^* + g^*$.
\item Suppose $$f(x)=g(x-a)+\langle x,a^*\rangle+\alpha\qquad \text{for all $x\in R^n$,}$$
where $a,a^*\in R^n$ and $\alpha\in R$. Then
$$f^*(x^*)=g^*(x^*-a^*)+\langle x^*,a\rangle+\alpha^*\qquad \text{for all $x^*\in R^n$,}$$
where $\alpha^*=-\alpha-\langle a,a^*\rangle$.
\end{enumerate}
\end{lemma}
\begin{proof}
Part (1) is easily verified by direct computation. For (2) see \cite[Theorem~2.3.1,~(ix)]{Zali}, and for (3) see \cite[Theorem~12.3]{Rockafellar}.
\end{proof}

If $g\colon R^n\to R_{\pm\infty}$ is definable and concave (so $-g$ is convex), then
the {\bf conjugate of $g$} is the definable function  $g^*\colon R^n\rightarrow R_{\pm\infty}$ given by
\[
g^*(x^*):=\inf_{x\in R^n}\big(\langle x,x^*\rangle-g(x)\big)=-(-g)^*(-x^*).
\]
Next we show a definable version of the 
Fenchel Duality Theorem \cite[Theorem~31.1]{Rockafellar} in a special case:

\begin{proposition}\label{prop:duality}
Let $f\colon R^n\to R_\infty$ be definable proper convex, and let $g\colon R^n\to R$ be definable continuous concave. Then
$$\inf_{x\in R^n} \big(f(x)-g(x)\big) = \max_{x^*\in R^n} \big(g^*(x^*)-f^*(x^*)\big).$$
\end{proposition}
\begin{proof}
For all $x,x^*\in R^n$ we have $$f(x)+f^*(x^*) \geq \langle x,x^* \rangle \geq g(x)+g^*(x^*)$$ by Fenchel-Young, hence $\inf_x \big(f(x)-g(x)\big) \geq \sup_{x^*}  \big(g^*(x^*)-f^*(x^*)\big)$. Set $\alpha:=\inf_x \big(f(x)-g(x)\big)$; we may assume $\alpha>-\infty$. It now suffices to show that there exists $x^*\in R^n$ such that $g^*(x^*)-f^*(x^*)\geq\alpha$. Consider the non-empty definable convex sets 
$$A:=\epi(f),\qquad B:=\big\{(x,t)\in R^{n+1} : t<g(x)+\alpha \big\}.$$ 
Then $B$ is open, and $A\cap B=\emptyset$. Hence by Corollary~\ref{cor:separation} there exists a hyperplane $H$ in $R^{n+1}$ separating $A$ and $B$. If $H$ were vertical, i.e., of the form $H=H'\times R$ for some hyperplane $H'$ in $R^n$, then $H'$ would separate $\dom(f)$ and $R^n$, which is impossible. Therefore $H$ is the graph of an affine function $x\mapsto \langle x,x^*\rangle-\alpha^*$ ($x^*\in R^n$, $\alpha^*\in R$).
Then for all $x\in R^n$ we have
$$f(x)\geq \langle x,x^*\rangle-\alpha^*\geq g(x)+\alpha.$$
This yields $\alpha=(\alpha^*+\alpha)-\alpha\leq g^*(x^*)-f^*(x^*)$ as required.
\end{proof}

\subsection{Examples of conjugates}
The functions discussed in the following examples will be of constant use below.

\begin{example}
The function $x\mapsto q(x):=\frac{1}{2}\norm{x}^2\colon R^n\to R$ is the only definable closed convex function $R^n\to R_\infty$ such that $q^*=q$.
\end{example}
\begin{proof}
To see that $q$ is convex use the identity
$$\norm{\lambda x + \mu y}^2 = 
\lambda\norm{x}^2 + \mu\norm{y}^2-\lambda\mu\norm{x-y}^2$$
which holds for all $x,y\in R^n$ and $\lambda,\mu\in R^{\geq 0}$ with $\lambda+\mu=1$. Since $q$ is continuous, $q$ is closed. Let $f\colon R^n\to R_\infty$ be definable closed convex such that $f^*=f$. Then $f$ is proper, and by Fenchel's Inequality $\langle x,x\rangle\leq f(x)+f^*(x)=2f(x)$, thus $f\geq q$ and hence $f=f^*\leq q^*=q$, so $f=q$.
\end{proof}

\begin{example}
The conjugate of the convex function $\kappa\colon R^n\times R^n\rightarrow R$ given by $\kappa(x,y):=q(x-y)$ is the function $\kappa^*\colon R^n\times R^n\rightarrow R_\infty$ given by
\[
\kappa^*(x^*,y^*)=\begin{cases}
q(x^*),&\text{if } x^*=-y^*,\\
+\infty,&\text{otherwise}.
\end{cases}
\]
\end{example}
\begin{proof}
Suppose that $x^*\neq -y^*$. Then $\norm{x^*+y^*}>0$, hence
\begin{align*}
\kappa^*(x^*,y^*)&=\sup_{(x,y)}{\textstyle\left(\big\langle(x^*,y^*),(x,y)\big\rangle-\frac{1}{2}\norm{x-y}^2\right)}\\
&\hskip-2.1em\underbrace{\geq}_{x=y=t(x^*+y^*)} \sup_{t} \big\langle(x^*,y^*), t(x^*+y^*,x^*+y^*)\big\rangle=\sup_t t\norm{x^*+y^*}^2=\infty.
\end{align*}
We also have  
\[\textstyle
\frac{1}{2}\norm{x^*}^2=\big\langle(x^*,-x^*),(x^*,0)\big\rangle-\frac{1}{2}\norm{x^*-0}^2\leq \kappa^*(x^*,-x^*)
\]
and
\begin{align*}
\kappa^*(x^*,-x^*)&=\sup_{(x,y)}{\textstyle\left(\big\langle(x^*,-x^*),(x,y)\big\rangle-\frac{1}{2}\norm{x-y}^2\right)}\\
&= \sup_{(x,y)}{\textstyle\left(\langle x^*,x-y\rangle-\frac{1}{2}\norm{x-y}^2\right)}\\
&\leq \sup_{(x,y)}{\textstyle\left(\norm{x^*}\norm{x-y}-\frac{1}{2}\norm{x-y}^2\right)}\\
&\hskip-2em\underbrace{=}_{\norm{x-y}=t\norm{x^*}}\sup_{t}{\textstyle\norm{x^*}^2\left(t-\frac{t^2}{2}\right)}
=\textstyle\frac{1}{2}\norm{x^*}^2,
\end{align*}
hence $\kappa^*(x^*,-x^*)=\frac{1}{2}\norm{x^*}^2$.
\end{proof}

\begin{example} \label{ex:Delta}
The function $\Delta\colon R^n\times R^n\to R$ given by
$\Delta(x,y) := q(x+y)=\kappa(x,-y)$ is convex and continuous.
Note that $\Delta$ satisfies the useful identity
$$\textstyle\Delta(x,y) = \frac{1}{2}\norm{x}^2 + \langle x,y\rangle + \frac{1}{2}\norm{y}^2.$$
Fix $(a,b)\in R^n\times R^n$ and define $\delta\colon R^n\times R^n\to R$ by
$$\delta(x,y) := \Delta(a-x,b-y)-\langle x,y\rangle.$$
Let $(x^*,y^*)\in R^n\times R^n$.
Then
$$\delta^*(x^*,y^*) = \delta(-y^*,-x^*).$$
\end{example}
\begin{proof}
We have
$$\delta(x,y) = q\big((x,y)-(a,b)\big) - \big\langle (x,y), (b,a) \big\rangle + \langle a,b\rangle$$
and hence by Lemma~\ref{lem:conjugates},~(3):
\begin{align*}
\delta^*(x^*,y^*) &= q^*\big( (x^*,y^*)+(b,a) \big) + \big\langle (x^*,y^*), (a,b) \big\rangle + \langle a,b\rangle \\
&= q\big( (-y^*,-x^*)-(a,b) \big) - \big\langle (-y^*,-x^*), (b,a) \big\rangle + \langle a,b\rangle =\delta(-y^*,-x^*).
\end{align*}
\end{proof} 

The following observations about $\kappa$ are used in the next subsection:

\begin{lemma}\label{lem:inf attained}
Let $g\colon R^n\to R_\infty$ be definable proper closed convex, and let $\lambda\in R^{>0}$ and $x\in R^n$. Then
\[\inf_y g(y)+\kappa(x,\lambda y)=\min_y g(y)+\kappa(x,\lambda y).\]
\end{lemma}
\begin{proof}
By Proposition~\ref{prop:closed convex}, there is an affine function $\varphi\colon R^n\rightarrow R$ such that $\varphi\leq g$. 
So the definable function $h\colon R^n\rightarrow R_\infty$, $h(y):=g(y)+\kappa(x,\lambda y)$ is closed convex such that 
$\lim_{\norm{y}\rightarrow +\infty}h(y)=+\infty$. 
Take some $z\in R^n$ with $g(z)<\infty$. Then $B:=\{y\in R^n:h(y)\leq h(z)\}$ is closed and bounded, and the continuous definable function $y\mapsto \varphi(y)+\kappa(x,y)$ attains a minimum on $B$. (Corollary~\ref{cor:minmax}.) Hence
the definable set
\[
\epi(h)\cap\{(y,t)\in R^n\times R:t\leq h(z)\}=\{(y,t)\in R^n\times R: h(y)\leq t\leq h(z)\}
\]
is non-empty, closed, and bounded. 
So is its projection on the last coordinate. (Proposition~\ref{prop:cbd}.) Hence $h$ attains its infimum.
\end{proof}

\begin{lemma}\label{lem:infconv}
Let $g\colon R^n\times R^n\to R_\infty$ be definable and proper closed convex. Then  $g\boxempty\kappa^*$ is proper closed convex.
\end{lemma}
\begin{proof}
For $(x,x^*)\in R^n\times R^n$, we have
$$(g\boxempty\kappa^*)(x,x^*) = \inf_{y\in R^n} g(x-y,x^*+y)+q(y),$$
and by the previous lemma, the infimum is attained, so $-\infty < g\boxempty \kappa^*\leq g$, showing that $g\boxempty \kappa^*$ is proper.
Set $C:=\epi(g)\times\epi(\kappa^*)\subseteq R^m$, where $m=2(2n+1)$. By \eqref{eq:epi of convolution}, it remains to show that the definable convex set 
\[
\epi(g)+\epi(\kappa^*)=\{y+z: (y,z)\in C\}
\]
is closed. (Recall from Section~\ref{sec:Minkowski} that the sum of two closed convex sets is not closed in general.)
Let $x\in \cl(\epi(g)+\epi(\kappa^*))$ and
 $\varepsilon>0$. The definable set
\[
C_\varepsilon:=\big\{(y,z)\in C: \norm{x-(y+z)}\leq \varepsilon\big\}
\]
is closed, convex, and non-empty.

\begin{claim} $C_\varepsilon$ is  bounded.
\end{claim}

\begin{proof}[Proof of the claim]
For $t>0$ let $S^m(t):=\{x\in R^m:\norm{x}=t\}$.
Assume for a contradiction that $C_\varepsilon$ is unbounded. Take an arbitrary $p=(y,z)\in C_\varepsilon$.
Then there is a definable unbounded subset $I\subseteq R^{>0}$ such that 
$(p+S^{m}(t))\cap C_\varepsilon\neq\emptyset$ for each $t\in I$. By weak definable choice (Lemma~\ref{lem:defchoice}), there is a definable function $\tilde{\gamma}\colon I\to C_\varepsilon$ with $\tilde{\gamma}(t)\in (p+S^{m}(t))\cap C_\varepsilon$ for all $t\in I$. 
Consider $\gamma\colon I\to S^{m}(1)$ defined by $\gamma(t):=\frac{1}{t}(\tilde{\gamma}(t)-p)$. By Proposition~\ref{prop:BW}, after replacing $I$ by a suitable unbounded definable subset, we may assume that $\gamma$ converges. Let $p'=(y',z'):=\lim_{I\ni t\to \infty} \gamma(t)\in S^m(1)$.

\medskip

Then for every $\lambda\geq 0$, we have $p+\lambda p'\in C_\varepsilon$.
Indeed, observe that for every $t\in I$ we have $[p,p+t\gamma(t)]\subseteq C_\varepsilon$. Suppose for a contradiction that $\lambda>0$ satisfies
\begin{equation}\label{eq:dist to Cepsilon}
\delta:=d(p+\lambda p', C_\varepsilon)>0.
\end{equation}
Take $t\in I$ such that $t\geq\lambda$ and $\norm{\gamma(t)-p'}<\delta/\lambda$. Then
$$d(p+\lambda p', C_\varepsilon) \leq \norm{p+\lambda p'-(p+\lambda\gamma(t))}=\lambda\norm{p'-\gamma(t)}<\delta,$$
which contradicts \eqref{eq:dist to Cepsilon}.

\medskip

So we have
$\norm{x-y-z-\lambda (y'+z')}\leq \varepsilon$ for every choice of $\lambda\geq 0$.
Hence, $y'=-z'$. Moreover, $y+\lambda y'\in\epi(g)$ and $z+\lambda z'\in \epi(\kappa^*)$ for every $\lambda\geq 0$.
But the only possible $z'$ is $z'=(0,\dots,0,t)$ for some $t>0$.
Therefore, $y'=(0,\dots,0,-t)$, which implies that $\epi(g)$ contains a vertical line.
This contradicts that $g$ is proper.
\end{proof}

By the claim,  $\{C_\varepsilon\}_{\varepsilon>0}$ is a monotone definable family of non-empty closed and bounded sets, so $\bigcap_{\varepsilon>0} C_\varepsilon\neq\emptyset$ by Lemma~\ref{lem:monotone}.
Hence there is $(y,z)\in\bigcap_{\varepsilon>0} C_\varepsilon\subseteq C$ such that $y+z=x$.
\end{proof}
\subsection{Proximal average}
Let $f,g\colon R^n\to R_\infty$ be definable. The definable function $\psi=\psi(f,g)\colon R^n\to R_{\pm\infty}$ given by
\[
\psi(x):=\inf_{y+z=x}\textstyle (\frac{1}{2}\ast f)(y)+(\frac{1}{2}\ast g)(z)+\kappa(y,z)
\]
is called the {\bf proximal average} of $f$ and $g$. This construction (cf.~\cite{Bauschke1, Bauschke2}) plays a key role in extending monotone set-valued maps in the next section. 

\begin{lemma}
Suppose $f$ and $g$ are proper closed convex. Then $\psi(f,g)$ is proper convex, with conjugate $\big(\psi(f,g)\big)^* = \psi(f^*,g^*)$.
\end{lemma}

\begin{proof}
Define $A\colon R^n\times R^n\to R^n$ by $A(y,z)=y+z$; then $A^*\colon R^n\to R^n\times R^n$ is given by $A^*(x^*)=(x^*,x^*)$. Also define the proper closed convex functions
$F,G\colon R^n\times R^n\to R_\infty$ by
$$\textstyle G(y,z)=(\frac{1}{2}\ast f)(y)+(\frac{1}{2}\ast g)(z),\quad F(y,z) = G(y,z)+\kappa(y,z).$$ 
So for each $x\in R^n$ we have
$$\psi(x)=(AF)(x)=\inf_y G(y,x-y)+\kappa(x,2y).$$
Hence $\psi$ is convex, and by Lemma~\ref{lem:inf attained} the infimum is attained, so $\psi$ is proper.
By Lemma~\ref{lem:infconv}, the definable convex function $G^*\boxempty \kappa^*$ is closed, hence
$$\textstyle F^*=\left(G+\kappa\right)^* = 
\left(G^{**}+\kappa^{**}\right)^* =
\left(G^*\boxempty\kappa^*\right)^{**} =  
G^* \boxempty \kappa^*.$$
Now for all $y^*,z^*\in R^n$,
$$\textstyle G^*(y^*,z^*)  = \frac{1}{2}f^*(y^*) + \frac{1}{2}g^*(z^*).$$
Hence for all $x^*\in R^n$,
\begin{align*} 
\big(\psi(f,g)\big)^*(x^*)   &= (AF)^*(x^*) \\
							&= F^*(A^*(x^*)) \\
							&= \textstyle\left(G^*\boxempty\kappa^*\right)(x^*,x^*)\\
							&= \inf_{(y^*,z^*)}\textstyle\big( G^*(y^*,z^*)+\kappa^*\left(x^*-y^*, x^*-z^*\right)\big)\\
							&= \inf_{y^*+z^*=2x^*}\textstyle\big(\frac{1}{2}f^*(y^*)+\frac{1}{2}g^*(z^*)+q\big(\frac{1}{2}(y^*-z^*)\big)\big)\\
							&= \big(\psi(f^*,g^*)\big)(x^*).
\end{align*}
\end{proof}

Let $f\colon R^n\times R^n\rightarrow R_{\pm\infty}$ be definable. We define the transpose $f^\trans$ of $f$ by
$f^\trans(x,x^*):=f(x^*,x)$ for all $(x,x^*)\in R^n\times R^n$.
We say that $f$ is {\bf autoconjugate} if $f^*=f^\trans$. Note that if $f$ is autoconjugate, then $f=f^{*\trans}$ is closed convex.

\begin{proposition}\label{prop:autoconjugate}
Let $f\colon R^n\times R^n\rightarrow R_\infty$ be definable proper closed convex.
Then the proximal average $\psi(f,f^{*\trans})\colon R^n\times R^n\to R_\infty$ of $f$ and $f^{*\trans}$ is autoconjugate.
\end{proposition}

\begin{proof}
Note that $f^{*\trans}=f^{\trans *}$ and hence $f^{*\trans*}=f^\trans$. So by the previous lemma,
$$\big(\psi(f,f^{*\trans})\big)^* = \psi(f^*,f^{*\trans*})  
= \psi(f^*, f^\trans) 
= \psi(f^\trans, f^*) 
= \psi(f^\trans, f^{*\trans\trans})
= \big(\psi(f,f^{*\trans})\big)^\trans.$$
\end{proof}

\begin{remark}
In the proof of the result analogous to Proposition~\ref{prop:autoconjugate} in \cite{Bauschke1}, appeals to more general results replace our use of the elementary Lemmas~\ref{lem:inf attained} and \ref{lem:infconv} above.
\end{remark}
\section{Proof of Theorem~\ref{thm:Kirszbraun}}\label{sec:Kirszbraun proof}
\noindent
Let $\mathfrak R$ be an expansion of a real closed ordered field.
In this section we prove Theo\-rem~\ref{thm:Kirszbraun}, which we state here again for the convenience of the reader, in a slightly strengthened form:

\begin{theorem}\label{DefKirsz}
Suppose $\mathfrak{R}$ is definably complete. Let $L\in R^{>0}$ and let $f\colon A\to B$, where $A\subseteq R^m$, $B\subseteq R^n$, be a definable $L$-Lipschitz map.
There exists a definable $L$-Lipschitz map $F\colon R^m\rightarrow \cl(\conv(B))$ such that $F|A=f$.
\end{theorem}
In fact, the extra condition $F(R^m)\subseteq\cl(\conv(B))$ is easy to achieve once we have a definable $L$-Lipschitz map $F'\colon R^m\to R^n$ with $F'|A=f$: simply take $F:=p\circ F'$ where $p=p({-},\cl(\conv(B)))$, and recall that $p$ is non-expansive by Corollary~\ref{cor:non-expansive}.

\medskip
Naturally, the question arises whether the hypothesis of definable completeness in this theorem is necessary. This question is affirmatively answered by the following proposition.
\begin{proposition}\label{prop:def completeness necessary}
Suppose $\mathfrak{R}$ is {\em not} definably complete. Then there exists a definable non-expansive function $f\colon A\to R$, where $A\subseteq R$ is closed, which cannot be extended to a non-expansive function $R\to R$.
\end{proposition}
\begin{proof}
Since $\mathfrak{R}$ is not definable complete, there exists a closed non-empty definable set $S\subseteq R$ which is bounded from above and which does not have a least upper bound in $R$. We let 
$$A_1:=\{a\in R : \text{$a\leq x$ for some $x\in S$}\},\qquad A_2:=R\setminus (1+A_1).$$
We have $S\subseteq A_1\subseteq 1+A_1 < A_2$.
Both $A_1$ and $A_2$ are closed, hence   $A:=A_1\cup A_2$ is a closed definable  subset of $R$.
After passing from $S$ to a suitable affine image $a+bS$ ($a,b\in R$), we may assume that $1+A_1\not\subseteq A_1$ and so $A\neq R$. 
 
Let $f\colon A\to R$ be defined by $f(x):=1$ if $x\in A_1$ and $f(x):=0$ if $x\in A_2$. Clearly, $f$ is definable and non-expansive.
Assume for a contradiction that there is a non-expansive $F\colon R\to R$ which extends $f$. Fix an arbitrary $x\in R\setminus A$;
then $x$ is an upper bound for $A_1$ and a lower bound for $A_2$. Hence, for all $y\in A_1$ and $z\in A_2$, we have 
\[
1+y-x=f(y)-\abs{y-x}\leq F(x)\leq f(z)+\abs{z-x}=z-x.
\]
So $\zeta:=F(x)+x$ is an upper bound for $1+A_1$ and a lower bound for $A_2$.
Thus $\zeta\not\in 1+A_1$ since $1+A_1$ has no least upper bound in $R$, and $\zeta\not\in A_2$ since $A_2$ has no largest lower bound in $R$, contradicting $R=(1+A_1)\cup A_2.$
\end{proof}

{\it In the rest of this section we assume that $\mathfrak R$ is definably complete.}

\medskip

We prove Theorem~\ref{DefKirsz} at the end of this section. In the rest of this subsection we mention two special cases of this theorem that are not hard to show directly. We let $f\colon A \to B$ be a definable map,  where $A$ is a non-empty subset of $R^m$ and $B\subseteq R^n$. First, Lemma~\ref{lem:inf and sup of Lipschitz functions} yields Theorem~\ref{thm:Kirszbraun} for a $1$-dimensional target space. More generally, we have the following result; here and below, a function $\omega\colon R^{\geq 0}\to R$ is said to be {\bf subadditive} if $\omega(s+t)\leq\omega(s)+\omega(t)$ for all $s,t\in R^{\geq 0}$. For example, it is easy to see that if $A$ is convex, then the modulus of continuity $\omega_f$ of $f$ is subadditive.

\begin{proposition}[McShane-Whitney]\label{prop:McShane-Whitney}
Suppose $n=1$ and $f$ has a definable increasing subadditive mod\-u\-lus of continuity $\omega$. Then
$$x\mapsto\inf_{a\in A} \big(f(a)+\omega\big(||x-a||\big)\big),\qquad x\mapsto\sup_{a\in A} \big(f(a)-\omega\big(||x-a||\big)\big)$$
are definable functions $R^n\to R$ extending $f$ with modulus of continuity $\omega$.
\end{proposition}

To prove this, by Lemma~\ref{lem:inf and sup of Lipschitz functions} one only needs to show that given $\omega$ as in the proposition, for each $a\in A$, the function $x\mapsto \omega\big(||x-a||\big)$ has modulus of continuity $\omega$, and this follows by a straightforward computation.

\medskip

Theorem~\ref{DefKirsz} for Lipschitz maps with convex domain is also easy to show:

\begin{proposition}\label{prop:Kirszbraun convex}
Suppose $A$ is convex, and $f$ is uniformly continuous \textup{(}$L$-Lip\-schitz, where $L\in R^{\geq 0}$\textup{)}. Then  there exists a definable map $F\colon R^m\to \cl(B)$ with $F|A=f$ which is   uniformly continuous \textup{(}$L$-Lipschitz, respectively\textup{)}. If $f$ is convex, then $F$ can additionally be chosen to be convex.
\end{proposition}

This is an immediate consequence of Lemma~\ref{lem:extension to closure} and the following lemma:

\begin{lemma}\label{lem:uniform extension convex}
Suppose $A$ is closed and convex. Then there exists a definable map $F\colon R^m\to B$ with $F|A=f$ and $\omega_f=\omega_F$. 
\end{lemma}
\begin{proof}
For $x\in R^m$ put $F(x):=f(p(x,A))$. Then the map $F\colon R^m\to B$ agrees with $f$ on $A$. Moreover, let $\delta>0$. Then for $x_1,x_2\in R^m$ with $\norm{x_1-x_2}\leq\delta$, setting $y_i=p(x_i,A)$ for $i=1,2$, we have $\norm{y_1-y_2}\leq\delta$ by Corollary~\ref{cor:non-expansive} and hence $\norm{F(x_1)-F(x_2)}=\norm{f(y_1)-f(y_2)}\leq\omega_f(\delta)$. This yields $\omega_F(\delta)\leq \omega_f(\delta)$, and the inequality $\omega_F(\delta)\geq \omega_f(\delta)$ is immediate. 
\end{proof}

\subsection{Monotone set-valued maps}
The crucial technique in proving Theorem \ref{thm:Kirszbraun} is to transfer the extension problem to definable monotone set-valued maps. As we will prove, these maps stay in one-to-one correspondence with firmly non-expansive maps. 
(See \cite{Phelps} for a useful survey on the theory of monotone set-valued maps in the context of Banach spaces.)

We begin by introducing (definable) set-valued maps as an alternative language for talking about (definable) families of sets. 
We use the notation $T\colon R^m\rightrightarrows R^n$ to denote a map $T\colon R^m\to 2^{R^n}$, and call such $T$ a {\bf set-valued map.} Such a set-valued map $T$ is {\bf trivial} if $T(x)=\emptyset$ for all $x\in R^m$.
The {\bf inverse} of a set-valued map $T\colon R^m\rightrightarrows R^n$ is the set-valued map $T^{-1}\colon R^n\rightrightarrows R^m$ given by
$$T^{-1}(x^*)=\big\{x\in R^m: x^*\in T(x)\big\}\qquad\text{for $x^*\in R^n$.}$$
Given set-valued maps $S,T\colon R^m\rightrightarrows R^n$ and $\lambda\in R$, the set-valued maps $S+T,\lambda\, S\colon R^m\rightrightarrows R^n$ are defined by
$(S+T)(x) = S(x)+T(x)$ and $(\lambda\, S)(x)=\lambda\,S(x)$ for $x\in R^m$.

Let $\mathcal T=(T_x)_{x\in X}$ be a family of subsets of $R^n$, where $X\subseteq R^m$. Then $\mathcal T$ gives rise to a set-valued map $T\colon R^m\rightrightarrows R^n$ by setting $T(x):=T_x$ for $x\in X$ and $T(x):=\emptyset$ for $x\in R^m\setminus X$. A set-valued map $R^m\rightrightarrows R^n$ arising in this way from a definable family $\mathcal T=(T_x)_{x\in X}$ of subsets of $R^n$ with $X\subseteq R^m$ is said to be {\bf definable.}

Let $T\colon R^m\rightrightarrows R^n$ be a set-valued map.
The {\bf graph} of $T$ is the subset $$\graph(T):=\big\{(x,x^*)\in R^m\times R^n: x^*\in T(x)\big\}$$
of $R^m\times R^n$. 
Note that every map $f\colon X\to R^n$, $X\subseteq R^m$, gives rise to a set-valued map $R^m\rightrightarrows R^n$, whose graph is the graph of the map $f$. We continue to denote the set-valued map associated to $f$ by the same symbol.
Given $S\colon R^m\rightrightarrows R^n$, we say that {\bf $T$ extends $S$} if $\graph(S)\subseteq \graph(T)$, and we say that {\bf $T$ properly extends $S$} if $\graph(S)\subsetneq \graph(T)$.

\begin{definition}
Let $T\colon R^n\rightrightarrows R^n$. 
An element $(x,x^*)\in R^n\times R^n$ is said to be {\bf monotonically related} to $T$ if 
\[
\langle x-y,x^*-y^*\rangle\geq 0 \qquad\text{for all $(y,y^*)\in \graph(T)$.}\] 
We say that $T$ is {\bf monotone} if every $(x,x^*)\in\graph(T)$ is monotonically related to $T$, and $T$ is called {\bf maximal monotone} if $T$ is monotone, and no $(x,x^*)\notin\graph(T)$ is monotonically related to $T$. (Equivalently, $T$ is maximal monotone if $T$ is monotone but every proper extension of $T$ fails to be monotone). 
\end{definition}

Clearly $T$ is monotone (maximal monotone) if and only if $T^{-1}$ is monotone (maximal monotone, respectively). It is easy to show that if $T\colon R^n\rightrightarrows R^n$ is maximal monotone, then $T(x)$ is a convex subset of $R^n$, for each $x\in R^n$.

\begin{example}
Let $f\colon X\to R^n$, where $X\subseteq R^n$. If $f$ is firmly non-expansive, then 
(the set-valued map associated to) $f$ is monotone. If $n=1$, then $f$ is monotone if and only if the function $f$ is increasing: $x\leq y\Rightarrow f(x)\leq f(y)$, for all $x,y\in X$.
\end{example}
\begin{example}
Let $T\colon R^n\to R^n$ be $R$-linear. Then $T$ is monotone if and only if $T$ is positive (i.e., $\langle T(x),x\rangle\geq 0$ for all $x\in R^n$), and in this case, $T$ is maximal monotone. (See \cite[Example~1.5~(b)]{Phelps}.)
\end{example}

Our interest in definable set-valued maps is motivated by the following fact; compare with \cite{Eckstein}.
Its proof makes crucial use of Theorem~\ref{thm:Helly} (the definable version of Helly's Theorem).
\begin{proposition}\label{maxmono}
Let $T\colon R^n\rightrightarrows R^n$, and let $f:=(T+\id)^{-1}$.
Then
\begin{enumerate}
\item $T$ is monotone if and only if $f$ is \textup{(}the set-valued map corresponding to\textup{)} a firmly non-expansive map $X\to R^n$, for some $X\subseteq R^n$;
\item if $f$ is a firmly non-expansive map $R^n\to R^n$, then $T$ is maximal monotone;
\item if $T$ is definable and maximal monotone, then $f$ is a firmly non-expansive map $R^n\to R^n$.
\end{enumerate}
\end{proposition}
\begin{proof}
We first note that the linear map $(x,x^*)\mapsto (x+x^*,x)$ restricts to a bijection $\graph(T)\to\graph(f)$ with inverse $(y,y^*)\mapsto (y^*,y-y^*)$. So if $T$ is monotone and $(x,x_i^*)\in\graph(f)$, where $i=1,2$, then $(x_i^*,x-x_i^*)\in\graph(T)$ and hence $$0\leq \langle x_1^*-x_2^*, (x-x_1^*)-(x-x_2^*)\rangle=-\norm{x_1^*-x_2^*}^2$$ by monotonicity of $T$, so $x_1^*=x_2^*$. Hence $f$ is the set-valued map corresponding to a map $X\to R^n$, where $X\subseteq R^n$.
Now (1) is a consequence of this observation and the following identity, valid for all $(x,x^*), (y,y^*)\in \graph(T)$:
\begin{gather*}
\big\langle f(x+x^*)-f(y+y^*), (x+x^*)-(y+y^*)\big\rangle-\norm{f(x+x^*)-f(y+y^*)}^2\\
=\big\langle x-y,(x+x^*)-(y+y^*)\big\rangle-\norm{x-y}^2
=\langle x-y,x^*-y^*\rangle.
\end{gather*}
For (2), suppose $f$ is  a firmly non-expansive map $R^n\to R^n$, and $S\colon R^n\rightrightarrows R^n$ is a monotone set-valued map extending $T$. Then $(S+\id)^{-1}$ is a set-valued map corresponding to a map (by (1)) which extends $f=(T+\id)^{-1}$, hence $S=T$. 
For  (3), suppose that $T$ is definable and monotone, and $X\neq R^n$.  Let $x\in R^n\setminus X$. 
Then  $f$ extends to a firmly non-expansive map $X\cup\{x\}\to R^n$ by Corollary~\ref{cor:extend non-expansive} and the remark following it. 
Hence $T$ can be properly extended to a monotone set-valued map $R^n\rightrightarrows R^n$, so $T$ is not maximal.
\end{proof}

Let $f\colon R^n\times R^n\to R_\infty$. The set-valued map $T\colon R^n\rightrightarrows R^n$ with
$$\graph(T)=\big\{(x,x^*)\in R^n\times R^n: f(x,x^*)=\langle x,x^*\rangle\big\}$$
is called the set-valued map {\bf represented by $f$.}
If $f$ is definable proper convex and autoconjugate, then
the Fenchel-Young Inequality implies $f(x,x^*)\geq \langle x,x^*\rangle$ and $f^*(x,x^*)\geq \langle x,x^*\rangle$ for all $x,x^*\in R^n$. Together with the next proposition (due to \cite{SZ} in the classical case), this yields that
autoconjugate functions represent maximal monotone maps:

\begin{proposition}\label{prop:SZ}
Let $f\colon R^n\times R^n\to R_\infty$ be definable proper convex, and let $T\colon R^n\rightrightarrows R^n$ be the set-valued map represented by $f$.
If $f(x,x^*)\geq\langle x,x^*\rangle$ for all $x,x^*\in R^n$, then $T$ is monotone, and if in addition $f^*(x,x^*)\geq\langle x,x^*\rangle$ for all $x,x^*\in R^n$, then $T$ is maximal monotone.
\end{proposition}
\begin{proof}
Suppose $f(x,x^*)\geq\langle x,x^*\rangle$ for all $x,x^*\in R^n$. Then for $(x,x^*),(y,y^*)\in\graph(T)$, using the convexity of $f$:
\begin{multline*}
\textstyle\frac{1}{2}\langle x,x^*\rangle + \frac{1}{2}\langle y,y^*\rangle =
\frac{1}{2}f(x,x^*) + \frac{1}{2}f(y,y^*) \geq \\ \textstyle f\left(\frac{1}{2}x+\frac{1}{2}y,\frac{1}{2}x^*+\frac{1}{2}y^*\right)\geq \left\langle \frac{1}{2}x+\frac{1}{2}y, \frac{1}{2}x^*+\frac{1}{2}y^*\right\rangle,
\end{multline*}
and this yields $\langle x-y, x^*-y^*\rangle\geq 0$.
Now assume $f^*(x,x^*)\geq\langle x,x^*\rangle$ for all $x,x^*\in R^n$, and let $(y,y^*)\in R^n\times R^n$ be monotonically related to $T$, i.e., $\langle y-x,y^*-x^*\rangle\geq 0$ for all $(x,x^*)\in\graph(T)$. From Example~\ref{ex:Delta} recall the notation $\Delta(x,y) = \frac{1}{2}\norm{x+y}^2$ for $x,y\in R^n$. By assumption and since $\Delta\geq 0$, with $g:=(f^*)^\trans$ we have
$$g(x,x^*)-\langle x,x^*\rangle+\Delta(y-x,y^*-x^*)\geq 0 \qquad\text{for all $(x,x^*)\in\graph(T_f)$.}$$
Hence by Proposition~\ref{prop:duality} and Example~\ref{ex:Delta} there exists $(x,x^*)\in R^n\times R^n$ such that
$$g^*(x^*,x)-\langle x,x^*\rangle+\Delta(y-x,y^*-x^*)\leq 0.$$
Since $g^*(x^*,x)=f^*(x,x^*)$ therefore 
$$\langle x,x^*\rangle\leq f^*(x,x^*)\leq \langle x,x^*\rangle-\Delta(y-x,y^*-x^*).$$
Hence $(x,x^*)\in\graph(T)$, thus $\langle y-x,y^*-x^*\rangle\geq 0$, and
$$\textstyle 0=\Delta(y-x,y^*-x^*)=\frac{1}{2}\norm{y-x}^2 + \langle y-x,y^*-x^*\rangle + \frac{1}{2}\norm{y^*-x^*}^2,$$
therefore $(y,y^*)=(x,x^*)\in\graph(T)$.
\end{proof}

\subsection{The Fitzpatrick function}
Let $T\colon R^n\rightrightarrows R^n$ be a non-trivial definable set-valued map.
The function $\Phi_T\colon R^n\times R^n\rightarrow R_\infty$ given by
$$\Phi_T(x,x^*):=\sup_{(a,a^*)\in \graph(T)} \big( \langle x,a^*\rangle + \langle a,x^*\rangle - \langle a,a^*\rangle\big)$$ 
is the {\bf Fitzpatrick function} of $T$. (This concept was introduced in \cite{Fitzpatrick}.) The function $\Phi_T$ is the pointwise supremum of a definable family of affine functions, hence $\Phi_T$ is definable and closed convex.
For $(x,x^*)\in R^n\times R^n$ we have
$$\Phi_T(x,x^*) =  
\langle x,x^*\rangle-\inf_{(a,a^*)\in \graph(T)}\langle x-a,x^*-a^*\rangle.$$
Hence if $T$ is monotone, then $\Phi_T(x,x^*)=\langle x,x^*\rangle$  for all $(x,x^*)\in\graph(T)$, in particular, $\Phi_T$ is proper; and if $T$ is maximal monotone, then $\Phi_T$ represents $T$.

{\it From now on until the end of this subsection we assume that $T$ is monotone.}\/
Then the set-valued map represented by $\Phi_T^*$ also extends $T$:

\begin{lemma}\label{lem:Phi conjugate}
For all $(y,y^*)\in\graph(T)$ we have
$\Phi_T^*(y^*,y)=\langle y^*,y\rangle$.
\end{lemma}

In the following we use tildes to denote elements of $R^{n}\times R^{n}$. Given $\tilde{x}\in R^{n}\times R^{n}$, we write $\tilde{x}=(x,x^*)$ where $x,x^*\in R^n$, and we put $\tilde{x}^\trans=(x^*,x)$. Below we will often use the identity
$$\textstyle\langle x,x^*\rangle = \frac{1}{2}\langle \tilde{x}, \tilde{x}^\trans\rangle\qquad (\tilde{x}=(x,x^*)\in R^n\times R^n).$$
For $\tilde{x}\in R^n\times R^n$ we have
$$ \Phi_T(\tilde{x})=\sup_{\tilde{a}\in\graph(T)}\textstyle \langle \tilde{x}, \tilde{a}^\trans\rangle - \frac{1}{2}\langle\tilde{a},\tilde{a}^\trans\rangle$$
and hence, for
$\tilde{y}\in R^n\times R^n$:
$$\Phi_T^*(\tilde{y})=\sup_{\tilde{x}} \big(\langle \tilde{x}, \tilde{y}\rangle - \Phi_T(\tilde{x})\big) = 
\sup_{\tilde{x}} \inf_{\tilde{a}\in\graph(T)}\textstyle\left( \langle \tilde{y} -\tilde{a}^\trans,\tilde{x} \rangle + \frac{1}{2}\langle \tilde{a}, \tilde{a}^\trans\rangle\right).$$

\begin{proof}[Proof of Lemma~\ref{lem:Phi conjugate}]
Let  $\tilde{y}\in \graph(T)$. Then $\Phi_T(\tilde{y})=\frac{1}{2}\langle\tilde{y}^\trans,\tilde{y}\rangle$, so the Fenchel-Young Inequality applied to $\Phi_T$ yields
$\Phi_T^*(\tilde{y}^\trans)\geq \frac{1}{2}\langle \tilde{y}^\trans,\tilde{y}\rangle$.
We also have
$$\Phi_T^*(\tilde{y}^\trans)= 
\sup_{\tilde{x}} \inf_{\tilde{a}\in\graph(T)}{\textstyle\left( \langle \tilde{y}^\trans -\tilde{a}^\trans,\tilde{x} \rangle + \frac{1}{2}\langle \tilde{a}, \tilde{a}^\trans\rangle\right)}\leq \sup_{\tilde{x}} \textstyle\frac{1}{2}\langle \tilde{y},\tilde{y}^\trans\rangle=\frac{1}{2}\langle \tilde{y},\tilde{y}^\trans\rangle.$$
Hence, $\Phi_T^*(\tilde{y}^\trans)= \frac{1}{2}\langle \tilde{y}^\trans,\tilde{y}\rangle$.
\end{proof}

Let $\Psi_T\colon R^n\times R^n\rightarrow R$ be the proximal average of $\Phi_T$ and $\Phi_T^*$, that is,
\[ 
\Psi_T(\tilde{x})=\inf_{\tilde{y}+\tilde{z}=2\tilde{x}}\textstyle\left(\frac{1}{2}\Phi_T(\tilde{y})+\frac{1}{2}\Phi_T^{*}(\tilde{z}^\trans)+\frac{1}{4}\kappa(\tilde {y},\tilde{z})\right)\quad\text{for $\tilde{x}\in R^n\times R^n$.}
\]
By Proposition~\ref{prop:autoconjugate}, the definable function $\Psi_T$ is proper convex and autoconjugate.

\begin{lemma}
Let $\tilde{x}\in\graph(T)$. Then $\Phi_T(\tilde{x})=\Psi_T(\tilde{x})$.
\end{lemma}
\begin{proof}
By the Fenchel-Young Inequality, we have on the one hand
$$2\Psi_T(\tilde{x})=\Psi_T^{*\trans}(\tilde{x})+\Psi_T(\tilde{x})\geq \langle\tilde{x},\tilde{x}^\trans\rangle.$$ On the other hand
$$2\Psi_T(\tilde{x})\leq \Phi_T(\tilde{x})+\Phi_T^*(\tilde{x}^\trans)=\langle\tilde{x},\tilde{x}^\trans\rangle.$$
So $\Psi_T(\tilde{x})=\frac{1}{2}\langle \tilde{x},\tilde{x}^\trans\rangle=\Phi_T(\tilde{x})$.
\end{proof}

By Proposition~\ref{prop:SZ} and the previous lemma,
we have the following adaptation of \cite[Theorem~5.7]{Bauschke2}:
\begin{proposition}\label{prop:extensionoperator}
The set-valued map $\overline{T}\colon R^n\rightrightarrows R^n$ represented by $\Psi_T$ is a definable maximal monotone extension of $T$.
\end{proposition}

%
%
%
%

We are now able to prove the definable version of the Kirszbraun Theorem.

\subsection{Proof of Theorem~\ref{DefKirsz}}
Let $A\subseteq R^m$ be non-empty, and let $f\colon A\to R^n$ be a definable $L$-Lipschitz function, where $L\in R^{>0}$. If $m<n$, then let $f_1\colon A\times R^{n-m}\to R^n$ be given by $f_1(x_1,\dots,x_n):=f(x_1,\dots,x_m)$.
If $n\leq m$, then set $f_1(x):=(f(x),0,\dots,0)\in R^m$. Note that $f$ extends to a definable $L$-Lipschitz map $R^m\to R^n$ if and only if $f_1$  extends to a definable $L$-Lipschitz map $R^k\to R^k$, where $k=\max\{m,n\}$. So after replacing $f$ by $f_1$, we may assume that $m=n$. Replacing $f$ by $f/L$, we may also assume that $f$ is non-expansive. By Proposition~\ref{prop44} the definable map $g:=\frac{1}{2}(\id + f)$ is firmly non-expansive, and it suffices to show that $g$ admits an extension to a definable firmly non-expansive map $R^n\to R^n$.
The definable set-valued map $T:=g^{-1}-\id\colon R^n\rightrightarrows R^n$ is monotone by Proposition~\ref{maxmono},~(1).
By Proposition~\ref{prop:extensionoperator} there is a definable maximal monotone $\overline{T}\colon R^n\rightrightarrows R^n$ extending $T$. By Proposition~\ref{maxmono},~(3), $G:=(\overline{T}+\id)^{-1}$ is the graph of a definable firmly non-expansive map $R^n\to R^n$ extending $g$ as required. \qed

\medskip

Inspection of the proof of Theorem~\ref{DefKirsz} given above exhibits a certain uniformity in the construction:

\begin{corollary} \label{cor:uniform Kirszbraun}
Let $a\mapsto L_a\colon A\to R^{\geq 0}$ be a definable function.
Let $\{f_a\}_{a\in A}$ be a definable family of maps $f_a\colon S_a\to R^n$, where $S_a\subseteq R^m$, such that each $f_a$ is $L_a$-Lipschitz. There exists a definable family $\{F_a\}_{a\in A}$ of maps $R^m\to R^n$, each $F_a$ being $L_a$-Lipschitz and extending $f_a$.
\end{corollary}

We finish this section with a question related to Theorem~\ref{thm:Kirszbraun}, to which we do not know the answer.
For a definable set $S\subseteq R^n$, let $\mathcal L_m(S)$ be the $R$-linear space of all definable Lipschitz maps $S\to R^m$, equipped with the seminorm
$$f\mapsto\abs{f} = \sup_{x\neq y} \frac{\norm{f(x)-f(y)}}{\norm{x-y}}.$$
Theorem~\ref{thm:Kirszbraun} shows the existence of a map $E\colon \mathcal L_m(S)\to\mathcal L_m(R^n)$ such that  for all $f\in \mathcal L_m(S)$, the map $E(f)$ extends $f$, and $\abs{E(f)}\leq\abs{f}$.

\begin{question}
Is there an \emph{$R$-linear} map $E\colon \mathcal L_m(S)\to\mathcal L_m(R^n)$ and some $C\in R$ such that for all $f\in \mathcal L_m(S)$,
$E(f)$ extends $f$, and $\abs{E(f)} \leq C\,\abs{f}$?
\end{question}
Note that since we do not require $C\leq 1$ (unlike in Theorem~\ref{thm:Kirszbraun}), it is enough to consider the case $m=1$. 
For $R=\R$ and o-minimal $\mathfrak R$, the answer to this question is positive, as shown in \cite{Pawlucki}.

\section{Some Variants}\label{sec:Variants}

\noindent
In this section we discuss a variant of Kirszbraun's Theorem for locally definable maps, and the problem of definably extending uniformly continuous maps, which is related to (but easier than) the problem of definably extending Lipschitz maps.

\subsection{Kirszbraun's Theorem for locally definable maps}\label{sec:locally definable}
Let $\mathfrak{R}$ be an expansion of the ordered field of real numbers. 
A set $S\subseteq \R^n$ is said to be {\bf locally definable} (in $\mathfrak R$) if for every $x\in \R^n$ there exists an open ball $B$ with center $x$ such that $B\cap S$ is definable.
A map $S\to\R^m$, where $S\subseteq \R^n$, is called locally definable if its graph is locally definable.
This notion encompasses both the subanalytic setting and Shiota's  notion \cite{Shiotageo} of $\mathfrak{X}$ families with axiom (v):

\pagebreak[2] 

\begin{examples} \mbox{}
\begin{enumerate}
\item A set $S\subseteq\R^n$ is locally definable in the expansion $\R_{\operatorname{an}}$ of the ordered field of real numbers by all restricted analytic functions if and only if $S$ is subanalytic (cf.~\cite[p.~507]{vdDM}).
\item Each $\mathfrak X$ family satisfying axiom (v) gives rise to an o-minimal expansion of the ordered field of reals with the property that the sets locally definable in this structure are precisely the sets in the given $\mathfrak X$ family (cf.~\cite{Schuermann}).
\end{enumerate}
\end{examples}

Many of the techniques used to prove the definable version of the Kirszbraun Theorem in the previous sections cannot be applied to locally definable maps and sets. 
In particular,  the intersection or union of a locally definable family of sets is not locally definable anymore in general, and
the pointwise infimum of a locally definable family of functions is also not necessarily locally definable.
However:

\begin{lemma}\label{lem:increasing sequence}
Suppose for each $\ell\in\N$ we are given a locally definable map $f_\ell\colon A_\ell\to\R^n$, where $A_\ell\subseteq\R^m$, such that $B_\ell(0)\subseteq A_\ell\subseteq A_{\ell+1}$ and $f_\ell=f_{\ell+1}|A_\ell$ for every $\ell$. Then the map $F\colon\R^m\to\R^n$ given by $F(x)=f_\ell(x)$, where $\ell$ is such that $x\in A_\ell$, is locally definable. Moreover, if each $f_\ell$ is $L$-Lipschitz, where  $L\in\R^{\geq 0}$, then $F$ is $L$-Lipschitz.
\end{lemma}

This observation together with Theorem~\ref{DefKirsz} does yield locally definable variants of the Kirszbraun Theorem. For this, we fix a locally definable  $L$-Lipschitz map $f\colon A\rightarrow \R^n$, where $L>0$ and $A\subseteq \R^m$ is non-empty.
\begin{corollary}
Suppoe  $f$ is bounded. Then $f$ extends to a bounded locally definable $L$-Lipschitz map $\R^m\to\R^n$.
\end{corollary}
\begin{proof}
By considering $x\mapsto C\cdot f(x+a)$ (for suitable $C\in\R^{>0}$ and arbitrary $a\in A$) in place of $f$, we may assume $L=1$, $f$ is bounded by $1$, and $0\in A$. For every $\ell\in \N$
we construct  a  locally definable non-expansive map $f_\ell\colon A_\ell:=\overline{B}_\ell(0)\cup A\to \overline{B}_1(0)$
such that for each $\ell$ we have $f_{\ell+1}=f_{\ell}$ on $B_\ell(0)$ and $f_\ell=f$ on $A$. Set $f_0:=f$. Suppose now that $\ell>0$ and the map $f_{\ell-1}$ has been constructed already. By
Theorem~\ref{DefKirsz}, there is a definable non-expansive function $g_\ell\colon\R^m\rightarrow \overline{B}_1(0)$ such that $g_\ell=f_{\ell-1}$ on $A_{\ell-1}\cap \overline{B}_{\ell+2}(0)$. For $x\in A_{\ell}$, set 
\[
f_{\ell}(x)=\begin{cases}
g_\ell(x)&\text{if } \norm{x}\leq \ell,\\
f(x)&\text{if }\norm{x}>\ell.
\end{cases}
\]
We claim that $f_{\ell}$ is non-expansive.
Suppose $x,y\in A_{\ell}$. 
If $\norm{x},\norm{y}\leq \ell$ or $\norm{x},\norm{y}>\ell$, then clearly $\norm{f_{\ell}(x)-f_{\ell}(y)}\leq \norm{x-y}$. 
Assume now that $\norm{x}\leq \ell$ and $\norm{y}>\ell$.
If $\norm{y}\leq \ell+2$, then 
\[
\norm{f_{\ell}(x)-f_{\ell}(y)}=\norm{g_\ell(x)-g_\ell(y)}\leq \norm{x-y}.
\]
If $\norm{y} > \ell+2$, then $\norm{x-y}>2$. Since $\norm{f_{\ell}}\leq 1$, we have 
\[
\norm{f_{\ell}(x)-f_{\ell}(y)}\leq \norm{f_{\ell}(x)}+\norm{f_{\ell}(y)}\leq 1+1< \norm{x-y}.
\]
Hence $f_{\ell}$ is non-expansive.
Now apply Lemma~\ref{lem:increasing sequence}.
\end{proof}
An enhancement of the previous proof leads to the following corollary.
\begin{corollary}
Let $\varepsilon>0$. Then $f$ extends to a  locally definable $(L+\varepsilon)$-Lipschitz map $\R^m\to\R^n$.
\end{corollary}
\begin{proof}
After replacing $f$ by $x\mapsto \frac{1}{L}\big(f(x+a)-f(a)\big)$, where $a\in A$ is arbitrary, we may assume that $L=1$ and $f(0)=0$. 
Let $(\varepsilon_\ell)_{\ell\in \N}$ be a strictly decreasing sequence of positive real numbers such that $\sum_{\ell}\varepsilon_\ell<\varepsilon$. 
For each $\ell\in\N$ let 
\[
L_\ell:=1+\sum_{p=0}^\ell\varepsilon_p.
\]
We construct for every $\ell\in \N$ a  locally definable $L_\ell$-Lipschitz function $f_\ell\colon A_\ell:=\overline{B}_\ell(0)\cup A\rightarrow \R^n$
such that for each $\ell$ we have $f_{\ell+1}=f_{\ell}$ on $\overline{B}_{\ell}(0)$ and $f_\ell=f$ on $A\cap B_\ell(0)$. 
Set $f_0:=f$. Suppose $\ell>0$ and $f_{\ell-1}$ has been constructed already.
Let 
\[
M_\ell=\sup\big\{\norm{f_{\ell-1}(x)}: x\in A_{\ell-1},\ \norm{x}\leq \ell\big\}.
\]
 Select $r_\ell>\ell$ so big such that 
$\varepsilon_\ell r_\ell\geq M_\ell +\ell L_\ell$.
 By Theorem~\ref{DefKirsz}, there is a definable $L_{\ell-1}$-Lipschitz map $g_\ell\colon\R^m\rightarrow \R^n$ such that $g_\ell=f_{\ell-1}$ on $A_\ell\cap \overline{B}_{r_\ell}(0)$. Define
$f_\ell\colon A_\ell\rightarrow \R^n$ by 
\[
f_\ell(x)=\begin{cases}
g_\ell(x)&\text{if $\norm{x}\leq \ell$,}\\
f(x)&\text{if $\norm{x}>\ell$.}
\end{cases}
\]
We claim that $f_\ell$ is $L_\ell$-Lipschitz.
Suppose $x,y\in A_\ell$. 
If $\norm{x},\norm{y}\leq \ell$ or $\norm{x},\norm{y}> \ell$, then $\norm{f_\ell(x)-f_\ell(y)}\leq L_\ell\norm{x-y}$ is evident. 
Assume now that $\norm{x}\leq \ell$ and $\norm{y} > \ell$.
If $\norm{y}\leq r_\ell$, then 
\[
\norm{f_\ell(x)-f_\ell(y)}=\norm{g_\ell(x)-g_\ell(y)}\leq L_\ell\norm{x-y}.
\]
 If $\norm{y} > r_\ell$, then we have 
\begin{align*}
\norm{f_\ell(x)-f_\ell(y)}&\leq \norm{f_\ell(x)}+\norm{f_\ell(y)}\\
&\leq M_\ell+L_{\ell-1}\norm{y}\\
&<\varepsilon_\ell \norm{y}-\ell L_\ell+L_{\ell-1}\norm{y}\\
&=  L_\ell\norm{y}-\ell L_\ell\\
&\leq L_\ell\norm{y}-\norm{x}L_\ell\\
&\leq L_\ell\norm{x-y}.
\end{align*}
Hence $f_\ell$ is $L_\ell$-Lipschitz.
Now apply Lemma~\ref{lem:increasing sequence}.
\end{proof}
The previous two corollaries raise the following question, the answer to which we do not know:

\begin{question}
Does the Kirszbraun Theorem hold for locally definable maps, i.e.: given a locally definable $L$-Lipschitz map $f\colon A\rightarrow \R^n$, where $A\subseteq\R^m$, does $f$ extend to a  locally definable $L$-Lipschitz map $\R^m\to\R^n$?
\end{question}

\subsection{Extending uniformly continuous maps}
In this subsection we let $\mathfrak R$ be a definably complete expansion of an ordered field.
Every Lipschitz map is uniformly continuous, so in light of Theorem~\ref{DefKirsz} it is natural to ask: {\it when does a definable uniformly continuous map $A\to R^n$, where $A\subseteq R^m$, extend to a uniformly continuous map $R^m\to R^n$?}\/ The aim of this subsection is to give a complete answer to this question (see Proposition~\ref{prop:GZ}), following \cite{Gruenbaum}, where this question was treated for $R=\R$ without definability requirements.

For this, let $f\colon A\to B$ be a definable map, where $A\subseteq R^m$ is non-empty and and $B\subseteq R^n$. We also assume that $A$ is closed. (Recall from Lemma~\ref{lem:extension to closure} that a definable uniformly continuous map always extends to a definable uniformly continuous map on the closure of its domain.) Note that $f$ is uniformly continuous if and only if each of the $n$ coordinate functions of $f$ is uniformly continuous, and similarly with ``continuous'' in place of ``uniformly continuous.'' Hence, in order to study the extendability of $f$ to a uniformly continuous (or merely continuous) map $R^m\to R^n$, we may further assume that $n=1$, which we do from now on.

Before we study uniformly continuous extensions, it is perhaps worth noting that if $f$ is continuous, then $f$ always extends to a definable continuous function on $R^m$:

\begin{lemma}[Definable Tietze Extension Theorem]\label{lem:Tietze}
Suppose $f$ is continuous. Then there exists a definable continuous function $F\colon R^m\to R$ with $F|A=f$.
\end{lemma}
\begin{proof}
First assume $B=(1,2)$. In this case one simply verifies that the definable function $F\colon R^m\to B$ with $F|A=f$ and
$$F(x) := \inf_{a\in A} f(a)\cdot\frac{d(x,a)}{d(x,A)}\qquad\text{for $x\in R^m\setminus A$}$$
is continuous. This well-known formula is due to Riesz (1923), and related to similar extension formulas by Hausdorff (1919) and Tietze (1915). For the general case, let $\tau$ be a definable homeomorphism $R\to (1,2)$, such as $$t\mapsto \frac{3}{2}+\frac{t}{2\sqrt{1+t^2}},$$ 
and note that if $F\colon R^m \to (1,2)$ extends $\tau\circ f$, then $\tau^{-1}\circ F$ extends $f$.
\end{proof}


(The proof of the definable version of Tietze Extension above is shorter and more elementary than the one in \cite[Chapter~8]{vdDries-Tame}, which is only valid for o-minimal $\mathfrak R$ and uses triangulations.)

\medskip

The classical counterpart of the following fact was proved in \cite{Gruenbaum}:

\begin{proposition}\label{prop:GZ}
Suppose $f$ is uniformly continuous.
The following are equivalent:
\begin{enumerate}
\item $f$ extends to a definable uniformly continuous function $R^m\to R$;
\item  $f$ has a definable subadditive modulus of continuity $\omega$ such that $\omega(t)\to 0$ as $t\to 0^+$;
\item $f$ has an affine modulus of continuity;
\item $f$ has a definable concave modulus of continuity $\omega$ with $\omega(t)\to 0$ as $t\to 0^+$.
\end{enumerate}
\end{proposition}

The implication (1)~$\Rightarrow$~(2) is clear. The implications (2)~$\Rightarrow$~(3) and (3)~$\Rightarrow$~(4) follow from the next two lemmas, for which we fix a definable function $\omega\colon R^{\geq 0}\to R^{\geq 0}$. The classical proof of the first lemma (as given in \cite{Gruenbaum}) uses the archimedean property of $\R$. 

\begin{lemma}
Suppose $\omega$ is subadditive and $\omega(t)\to 0$ as $t\to 0^+$. Then there is an affine function $\omega_1\colon R\to R$ with $\omega_1\geq\omega$. 
\end{lemma}
\begin{proof}
There is $\delta>0$ such that
$\omega(t)<1$ for all $t\in [0,\delta]$, hence $\omega(t)\leq 2$ for $t\in[0,2\delta]$ by subadditivity.
We claim that $\omega_1\colon R\rightarrow R$ given by
$\omega_1(t):=2+\textstyle\frac{1}{\delta}t$
majorizes $\omega$.
Assume for a contradiction that the subset $B:=\{\omega>\omega_1\}$ of $R^{\geq 0}$
is non-empty, and put $b:=\inf B$. Evidently, we have $b\geq 2\delta$.
Let $s\in[b,b+\delta)$. Then
\begin{align*}
\omega(s)&=\omega(s-\delta+\delta)\leq\omega(s-\delta)+\omega(\delta)\\
&< \omega_1(s-\delta)+1=2+\textstyle\frac{1}{\delta}(s-\delta)+1=\omega_1(s).
\end{align*}
Hence $B\cap [b,b+\delta)=\emptyset$, a contradiction.
\end{proof}

\begin{lemma}[McShane]
Suppose there exists an affine function $\omega_1$ with $\omega_1\geq\omega$. Then there exists a definable  concave function $\omega_2$ with $\omega_2\geq\omega$; if $\omega(t)\to 0$ as $t\to 0^+$, then $\omega_2$ can be chosen so that moreover $\omega_2(t)\to 0$ as $t\to 0^+$.
\end{lemma}
\begin{proof}
For $a,b\in R$  let
$\omega_{a,b}(t)=a+bt$, and let $a_0,b_0\in R$ with $\omega_1=\omega_{a_0,b_0}$. Then
\[
\omega_2(t):=\inf\big\{\omega_{a,b}(t):  a,b\in R,\ \omega\leq \omega_{a,b}\big\}
\]
is a definable concave function with $\omega_2\geq\omega$. Assume now that $\lim\limits_{t\to 0^+}\omega(t) = 0$.
To see that $\lim\limits_{t\to 0^+}\omega_2(t)=0$, let $\varepsilon>0$ be given. Take $\delta>0$ such that $\omega(t)\leq\varepsilon$ for $0\leq t\leq\delta$. Take some $b>0$ such that $\omega_{\varepsilon,b}(t)>\omega_{a_0,b_0}(t)$ for $t>\delta$. Then $\omega(t)\leq\omega_{\varepsilon,b}(t)$ for all $t\geq 0$, so $\omega_2(t)\leq\omega_{\varepsilon,b}(t)$ for all $t\geq 0$. Also, $\omega_{\varepsilon,b}(t)\to\varepsilon$ as $t\to 0^+$. This yields the claim.
\end{proof}

The implication (4)~$\Rightarrow$~(1) in Proposition~\ref{prop:GZ} is a consequence of Proposition~\ref{prop:McShane-Whitney} and the next lemma:

\begin{lemma}
Let $\omega\colon R^{\geq 0}\to R^{\geq 0}$ be a concave function. Then $\omega$ is increasing, and if in addition $\omega(0)=0$, then $\omega$ is subadditive.
\end{lemma}
\begin{proof}
Suppose that $s$, $t$ are positive elements of $R$ such that $s<t$ and $\omega(s)>\omega(t)$. Put $\Delta:=t-s$, and choose $\lambda$ with
$0<\lambda<1$ and $(1-\lambda)\omega(s) > \omega(t)$.
Then
we have, by concavity of $\omega$:
$$\omega(t)=\omega(s+\Delta) \geq \lambda\omega\left(s+\textstyle\frac{1}{\lambda}\Delta\right)+(1-\lambda)\omega(s)$$
and hence
$$\omega\left(s+\textstyle\frac{1}{\lambda}\Delta\right)\leq\textstyle
\frac{1}{\lambda}\left(\omega(t)-(1-\lambda)\omega(s)\right)<0,$$
a contradiction. Hence $\omega$ is subadditive. If $\omega(0)=0$, note that for $s,t>0$, by concavity  $\omega(s)\geq\frac{s}{s+t}\omega(s+t)$, and similarly for $t$ in place of $s$; now add.
\end{proof}

\begin{corollary}
If $f$ is bounded and uniformly continuous, then there is a definable uniformly continuous function on $R^m$ extending $f$. In particular, if $A$ is bounded and $f$ is continuous, then $f$ extends to a definable uniformly continuous function on $R^m$.
\end{corollary}
\begin{proof}
If $M\in R$ is such that $\norm{f}\leq M$, then $\omega_f\leq 2M$. Hence the first statement follows from (3)~$\Rightarrow$~(1) in Proposition~\ref{prop:GZ}.
The second statement now follows from the first by Lemma~\ref{lem:cbd implies uniform continuous}.
\end{proof}

\begin{remarks}
Suppose $f$ is uniformly continuous.
\begin{enumerate}
\item If $B=[1,2]$, then the extension $F$ of $f$ to a function on $R^m$ defined as in the proof of Lemma~\ref{lem:Tietze} is also uniformly continuous. (This is shown for $R=\R$ in
\cite{Mandelkern}, and the proof given there goes through in general.)
\item If $A$ is convex, then $f$ extends to a definable uniformly continuous function on $R^m$ with the same modulus of continuity, cf.~Lemma~\ref{lem:uniform extension convex}. In \cite{Kleiner} it is shown that given a closed subset $S$ of $\R^m$, each uniformly continuous function on $S$ has an extension to a function on $\R^m$ with the same modulus of continuity if and only if $S$ is convex. 
\end{enumerate}
\end{remarks}

\bibliographystyle{amsplain} 

\begin{thebibliography}{10}


\bibitem{Bauschke1} H. H. Bauschke and X. Wang,
	{\it Firmly nonexpansive and Kirszbraun-Valentine extensions: a constructive approach via monotone operator theory,} in:
	{\it Nonlinear Analysis and Optimization \textup{(}Haifa 2008\textup{)},} Contemp. Math., Amer. Math. Soc., to appear.

\bibitem{Bauschke2} \bysame,	{\it The kernel average for two convex functions and its applications to the extension and representation of monotone operators,}
Trans. Amer. Math. Soc. {\bf 361} (2009), no. 11, 5947--5965. 
	
\bibitem{BenLi}
Y. Benyamini and J. Lindenstrauss, {\it Geometric Nonlinear Functional Analysis, Vol. 1,} Amer. Math. Soc. Colloquium Publications, vol. 48, Amer.
Math. Soc., Providence, RI, 2000. 



\bibitem{BL} 
J. M. Borwein and A. S. Lewis, {\it
Convex Analysis and Nonlinear Optimization,}
2nd ed., CMS Books in Mathematics/Ouvrages de Math\'ematiques de la SMC, vol. 3, Springer-Verlag, New York, 2006.


\bibitem{DGK}
L. Danzer, B. Gr\"unbaum, and V. Klee, 
{\it Helly's theorem and its relatives,} in: \cite{Klee}, 101--180.


\bibitem{Dolich} A. Dolich, 
{\it Forking and independence in o-minimal theories,}
J. Symbolic Logic {\bf 69} (2004), no. 1, 215--240.

\bibitem{DM}
R. Dougherty and C. Miller, 
{\it Definable Boolean combinations of open sets are Boolean combinations of open definable sets,}
Illinois J. Math. {\bf 45} (2001), no. 4, 1347--1350.

\bibitem{vdDries-Tame}
L. van den Dries, 
{\it Tame Topology and O-minimal Structures,} 
London Mathematical Society Lecture Note Series, vol. 248, Cambridge University Press, Cambridge, 1998. 

\bibitem{vdDries-Dense}
\bysame, {\it Dense pairs of o-minimal structures,} Fund. Math. {\bf 157} (1998), no. 1, 61--78. 

\bibitem{vdDM}
L. van den Dries and C. Miller, 
{\it Geometric categories and o-minimal structures,}
Duke Math. J. {\bf 84} (1996), no. 2, 497--540. 


\bibitem{Eckhoff}
J. Eckhoff,
{\it Helly, Radon, and Carath\'eodory type theorems,} in: 
P. M. Gruber and J. M. Wills (eds.), {\it Handbook of Convex Geometry,} Vol. A, 389--448, North-Holland Publishing Co., Amsterdam, 1993.

\bibitem{Eckstein} J. Eckstein and D. P. Bertsekas, 
	{\it On the Douglas-Rachford splitting method and the proximal point algorithm for maximal monotone operators,}
	Math. Programming {\bf 55} (1992), no. 3, Ser. A, 293--318.
	
\bibitem{Federer}
H. Federer, {\it Geometric Measure Theory,}
Die Grundlehren der mathematischen Wissen\-schaf\-ten, vol. 153, Springer-Verlag, New York, 1969.

\bibitem{Fitzpatrick}
S. Fitzpatrick, {\it Representing monotone operators by convex functions,} in: 
S. Fitzpatrick and J. Giles (eds.), {\it Workshop/Miniconference on Functional Analysis and Optimization \textup{(}Canberra,  1988\textup{)}}, 59--65, Proceedings of the Centre for Mathematical Analysis, Australian National University, vol. 20, Australian National University, Centre for Mathematical Analysis, Canberra, 1988.

\bibitem{Goebel} K. Goebel and W.A. Kirk,
	{\it Topics in Metric Fixed Point Theory}, Cambridge Studies in Advanced Mathematics, vol. 28, Cambridge University Press, Cambridge, 1990.

\bibitem{GD}
A. Granas and J. Dugundji, 
{\it Fixed Point Theory,}
Springer Monographs in Mathematics, Springer-Verlag, New York, 2003.
	

\bibitem{GL}
A. Granas and M. Lassonde, 
{\it Sur un principe g\'eom\'etrique en analyse convexe,}
Studia Math. {\bf 101} (1991), no. 1, 1--18. 

	
	
\bibitem{Gromov} M. Gromov, {\it Monotonicity of the volume of intersection of balls,} in: J. Lindenstrauss and V. D. Milman (eds.), {\it Geometrical Aspects of Functional Analysis,} 1--4, Lecture Notes in Mathematics, vol. 1267,  Springer-Verlag, Berlin, 1987. 

\bibitem{Gruenbaum} F. Gr\"unbaum and E. H. Zarantonello, 
	{\it On the extension of uniformly continuous mappings.}
	Michigan Math. J. {\bf 15} (1968), 65--74.


\bibitem{Heinonen} 
J. Heinonen, {\it Lectures on Lipschitz Analysis,} University of Jyv\"askyl\"a Department of Mathematics and Statistics Report, vol. 100, University of Jyv\"askyl\"a, Jyv\"askyl\"a, 2005.

	
\bibitem{HU-L}
J.-B.~Hiriart-Urruty and C. Lemar\'echal, {\it Convex Analysis and Minimization Algorithms. \textup{I.} 
Fundamentals,} Grundlehren der Mathematischen Wissenschaften, vol. 305, Springer-Verlag, Berlin, 1993.	

\bibitem{HP}	
E. Hrushovski and Y. Peterzil, {\it A question of van den Dries and a theorem of Lipshitz and Robinson; not everything is standard,} J. Symbolic Logic {\bf 72} (2007), no. 1, 119--122. 



\bibitem{Klee}
V. Klee (ed.), {\it Convexity \textup{(}Seattle, 1961\textup{)}}, Proc. Sympos. Pure Math., vol. VII, Amer. Math. Soc., Providence, RI, 1963. 

\bibitem{Kleiner} G. Kleiner, {\it Convex sets and the modulus of continuity,}
Zeszyty Nauk. Uniw. Jagiello. Prace Mat. Zeszyt {\bf 13} (1969), 41--44. 


\bibitem{Mandelkern}
M. Mandelkern, 
{\it On the uniform continuity of Tietze extensions,}
Arch. Math. (Basel) {\bf 55} (1990), no. 4, 387--388.



\bibitem{Miller} C. Miller,
{\it Expansions of dense linear orders with the intermediate value property,} 
J. Symbolic Logic {\bf 66} (2001), no. 4, 1783--1790. 

\bibitem{MS}
C. Miller and P. Speissegger, 
{\it Expansions of the real line by open sets: o-minimality and open cores,}
Fund. Math. {\bf 162} (1999), no. 3, 193--208. 

\bibitem{Minty}
G. Minty, {\it Monotone \textup{(}nonlinear\textup{)} operators in Hilbert space,}
Duke Math. J. {\bf 29}  (1962), 341--346. 

\bibitem{Pawlucki}
W. Paw\l{}ucki, {\it A linear extension operator for Whitney fields on closed o-minimal sets,} Ann. Inst. Fourier (Grenoble) {\bf 58} (2008), no. 2, 383--404.

\bibitem{PP} Y. Peterzil and A. Pillay, {\it Generic sets in definably compact groups,} Fund. Math. {\bf 193} (2007), 153--170.


\bibitem{Phelps}
R. R. Phelps, {\it  Lectures on maximal monotone operators,}
Extracta Math. {\bf 12} (1997), no. 3, 193--230.


\bibitem{Robson} 	
R. Robson, {\it Separating points from closed convex sets over ordered fields and a metric for $\tilde R{}\sp n$,} Trans. Amer. Math. Soc. {\bf 326} (1991), no. 1, 89--99.


\bibitem{Rockafellar} R. T. Rockafellar, {\it Convex Analysis,}
	Princeton Mathematical Series, vol. 28, Princeton University Press, Princeton, N.~J., 1970.


\bibitem{Sandgren} L. Sandgren, {\it On convex cones,} Math. Scand. {\bf 2} (1954), 19--28. 

\bibitem{Schneider}
R. Schneider, {\it Convex Bodies: The Brunn-Minkowski Theory,} Encyclopedia of Mathematics and its Applications, vol. 44, Cambridge University Press, Cambridge, 1993.

\bibitem{Schuermann}
J. Sch\"urmann, {\it On the comparison of different notions of geometric categories,} preprint (2002).

\bibitem{Shiotageo} M. Shiota, 
{\it Geometry of Subanalytic and Semialgebraic Sets,}
Progress in Mathematics, vol. 150, Birkh\"auser Boston, Inc., Boston, MA, 1997. 

\bibitem{SZ}
S. Simons and C. Z{\u{a}}linescu, 
{\it Fenchel duality, Fitzpatrick functions and maximal monotonicity,} J. Nonlinear Convex Anal. {\bf 6} (2005), no. 1, 1--22. 

	
\bibitem{Valentine}
F. A. Valentine, 
{\it The dual cone and Helly type theorems,} in \cite{Klee},  473--493.

\bibitem{Webster}
R. Webster,  {\it Convexity,} Oxford Science Publications,  Oxford University Press, New York, 1994.


\bibitem{Zali}
C. Z{\u{a}}linescu,
{\it Convex Analysis in General Vector Spaces,} World Scientific Publishing Co., Inc., River Edge, NJ, 2002.

\end{thebibliography}

\providecommand{\bysame}{\leavevmode\hbox to3em{\hrulefill}\thinspace}

\end{document}